\documentclass{amsart}
\usepackage[utf8]{inputenc}
\usepackage{hyperref}
\usepackage{amsmath, amsthm, amssymb}
\usepackage{tikz}
\usepackage{tikz-cd}
\usepackage{tikz-3dplot}
\usepackage{euscript}
\usepackage[maxbibnames=99]{biblatex}
\usepackage{lipsum}

\usepackage{fdsymbol}

\theoremstyle{plain}
\newtheorem{theorem}{Theorem}[section]
\newtheorem{corollary}[theorem]{Corollary}
\newtheorem{proposition}[theorem]{Proposition}
\newtheorem{lemma}[theorem]{Lemma}
\newtheorem*{theorem*}{Theorem}

\theoremstyle{definition}

\newtheorem{remark}[theorem]{Remark}
\newtheorem{definition}[theorem]{Definition}
\newtheorem*{definition*}{Definition}

\newtheorem{example}[theorem]{Example}
\hypersetup{%
	colorlinks,%
	linkcolor={red!60!black},%
	citecolor={red!60!black},%
	urlcolor={red!60!black}%
}

\def\on{\operatorname}

\def\bb{\mathbb}

\def\cal{\mathcal}
\def\PSL{\on{PSL}}
\def\SL{\on{SL}}

\def\GL{\on{GL}}
\def\Aut{\on{Aut}}
\def\Fit{\on{Fit}}
\def\SL{\on{SL}}
\def\Out{\on{Out}}

\def\C{\bb{C}}

\def\F{\bb{F}}

\def\Inn{\on{Inn}}
\def\ol{\overline}
\def\ul{\underline}

\def\graphcomp{\ol{\Gamma}}
\def\digraph{\overrightarrow{\Gamma}}

\newcommand{\edge}[2]{ {#1 \! - \! #2} }

\newcommand{\primetri}{
\begin{tikzpicture}[scale=0.2]
\draw (-1, 0) -- (0, -1.73) -- (1, 0) -- cycle;

\draw[fill=black] (-1, 0) circle (0.25);
\draw[fill=black] (0, -1.73) circle (0.25);
\draw[fill=black] (1, 0) circle (0.25);
\end{tikzpicture}
}

\newcommand{\primeline}{
\begin{tikzpicture}[scale=0.2]
\draw (-1, 0) -- (0, -1.73) -- (1, 0);

\draw[fill=black] (-1, 0) circle (0.25);
\draw[fill=black] (0, -1.73) circle (0.25);
\draw[fill=black] (1, 0) circle (0.25);
\end{tikzpicture}
}

\title[Prime Graphs of Groups with $K_3$ or Cyclic Composition Factors]{The Prime Graphs of Groups with Arithmetically Small Composition Factors}

\author[Edwards, Keller, Pesak, Sellakumaran Latha]{Timothy J. Edwards, Thomas Michael Keller, Ryan M. Pesak, Karthik Sellakumaran Latha}
\date{}

\begin{document}

\maketitle

\begin{abstract}
    We continue the study of prime graphs of finite groups, also known as Gruenberg-Kegel graphs. The vertices of the prime graph of a finite group are the prime divisors of the group order, and two vertices $p$ and $q$ are connected by an edge if and only if there is an element of order $pq$ in the group. Prime graphs of solvable groups have been characterized in graph theoretical terms only, as have been the prime graphs of groups whose only nonsolvable composition factor is $A_5$. In this paper we classify the prime graphs of all groups whose composition factors have arithmetically small orders, that is, have no more than three prime divisors in their orders. We find that all such graphs have $3$-colorable complements, and we provide full characterizations of the prime graphs of such groups based on the exact type and multiplicity of the nonabelian composition factors of the group.
\end{abstract}

\section{Introduction}

This paper continues the investigation of the prime graphs of finite groups, also called Gruenberg-Kegel graphs. The prime graph of a group is the graph whose vertices are the prime divisors of the group order, where two vertices $p$ and $q$ are connected if and only if there is an element of order $pq$. The notion was introduced by Gruenberg and Kegel in the 1970s, and prime graphs have been extensively studied ever since. For example, in \cite{PrimeGraphsAlmostSimple}, Maslova and Gorshkov classified all almost simple groups with prime graphs equal to the prime graph of a solvable group. It is also interesting to look at a given graph $\Gamma$ and ask whether groups exist whose prime graphs are isomorphic to $\Gamma$. In \cite{complete_bipartite_prime_graphs}, Maslova and Pagon examined complete bipartite graphs with respect to this question. More recently, Cameron \cite{cameron} noticed that prime graphs also have strong connections to a number of other graphs defined for finite groups, and the work of Qian, Wang, and Wei \cite{qian} shows that the prime graph is always a subgraph of the co-degree graph for that group.\\

In 2015, Gruber, Keller, Lewis, Naughton, and Strasser studied the prime graphs of solvable groups. The main result of that paper is the following complete characterization:

\begin{lemma}\label{lem:solvablegroups} 
\cite[Theorem 2.10]{solvable_groups} A simple graph $\Gamma$ is isomorphic to the prime graph of a solvable group if and only if its complement $\graphcomp$ is $3$-colorable and triangle-free.
\end{lemma}

This result was expanded upon in \cite{A5_paper} and \cite{A5classification} to classify the prime graphs of what the authors define as pseudo-solvable groups: groups whose composition factors are either cyclic or $A_5$. Groups with a single non-abelian composition factor were also investigated in \cite{1compositionfactor} with respect to their prime graphs. In this paper, we seek to expand on these results to classify the prime graphs of groups whose composition factors are either cyclic or another fixed simple group $T$. To this end, we introduce some new definitions.


\begin{definition}
Let $G$ be a group and $\mathcal{T}$ a set of distinct (isomorphism types of) nonabelian simple groups. We say that $G$ is \textit{pseudo $\mathcal{T}$-solvable} if each of its composition factors is either abelian or isomorphic to an element in $\mathcal{T}$. We say that $G$ is \textit{strictly pseudo $\mathcal{T}$-solvable} if $G$ is pseudo $\mathcal{T}$-solvable and has each element of $\mathcal{T}$ as at least one of its composition factors. If $\mathcal{T}$ contains only one element $T$, then we will abbreviate pseudo $\{T$\}-solvable to pseudo $T$-solvable.
\end{definition}

In this language, \cite{A5_paper} and \cite{A5classification} classify the prime graphs of pseudo $A_5$-solvable groups. A natural next step is to consider prime graphs of pseudo $T$-solvable groups for other simple nonabelian groups $T$ whose orders have exactly three prime divisors. Such simple nonabelian groups are called $K_3$-groups, and as \cite{K3_groups} reports, there are a total of eight such groups: $A_5, \PSL(2,7), A_6, \PSL(2,8), \PSL(2,17), \PSL(3,3)$, and the unitary groups $U_3(3)$ and $U_4(2)$. Some relevant information about each group is listed in Table \ref{table:1}. We denote the set (or family) of $K_3$-groups by $\mathcal{K}_3$. \\

The main result of this paper is a complete characterization of the prime graphs of all pseudo $\mathcal{K}_3$-solvable groups. It turns out that, in a way, "most" of them have prime graphs that are also prime graphs of solvable groups, as the following result shows.

\begin{theorem} \label{theorem1.3} (see Theorem \ref{thm:lookslikesolvable} below)
Let $G$ be a pseudo $\mathcal{K}_3$-solvable group, and suppose one of the following holds:
\begin{enumerate}
    \item $G$ has no nonabelian composition factors (i.e, $G$ is solvable);
    \item $G$ has exactly one nonabelian composition factor, which is $\PSL(3,3), U_3(3)$, or $U_4(2)$; or
    \item $G$ has at least two (not necessarily distinct) nonabelian composition factors.
\end{enumerate}
Then, $\graphcomp(G)$ is triangle-free and $3$-colorable.
\end{theorem}

All the more interesting are then those prime graphs of pseudo $\mathcal{K}_3$-solvable groups which are not also prime graphs of some solvable group. These are completely characterized in the other main results of this paper, but they are too numerous to list them all here in the introduction. The paper is organized as follows. \\

In Sections \ref{section:psl27}, \ref{section:a6}, \ref{section:psl28}, \ref{section:psl33}, and \ref{section:psl217}, we fully classify prime graphs of
pseudo $T$-solvable groups in the cases where $T$ is $\PSL(2,7)$, $U_3(3)$, $A_6$, $U_4(2)$, $\PSL(2,8)$, $\PSL(3,3)$, or $\PSL(2,17)$. We find that in the cases of $\PSL(2,7)$, $\PSL(2,8)$, $\PSL(2,17)$, and $A_6$, there can be at most one triangle in the complement of the prime graph. We also establish restrictions on how this triangle can connect to the rest of the graph. As an example of our results, we obtain the following classification for $T = A_6$:

\begin{theorem} (see Theorem \ref{thm:A6-classification} below)
Let $\Gamma$ be a graph. Then the following are equivalent:
\begin{enumerate}
    \item $\Gamma$ is the prime graph of a pseudo $A_6$-solvable group $G$.
    \item $\graphcomp$ is $3$-colorable and has at most one triangle, which must be isolated if it exists.
\end{enumerate}
\end{theorem}

In the cases of $\PSL(3,3)$, $U_3(3)$, and $U_4(2)$, we find that the complement of the prime graph must be $3$-colorable and triangle-free, just like for solvable groups. Larger groups such as the unitary groups are easier to study because they have many convenient subgroups. \\
So together with \cite{A5classification} we now have a complete classification of the prime graphs of all pseudo $T$-solvable groups for any $K_3$-group $T$. Note, however, that the results obtained in this paper, which covers all $K_3$-groups except for $A_5$, are in stark contrast to the results for $A_5$ in that the prime graphs in case of $A_5$ are much richer in structure and, in particular, their complements can have an arbitrary number of triangles. The main reason for this difference is the well-known fact that $A_5$ is the only non-abelian simple group which can be a section in a Frobenius complement. Thus in the guise of $\SL(2,5)$, $A_5$ can act frobeniusly on other groups, and this is what makes the case of $A_5$ so different. This does not mean, however, that the other $K_3$-groups are all easy to handle. Some  are, but others pose their very own challenges. To handle $\PSL(2,8)$, for example, we need to invoke a deep result by Flavell \cite{Flavell}, and no other $K_3$-group requires us to pull out such a "big gun".\\
Finally, in Section \ref{section:multiple-Ts} we consider groups with multiple non-isomorphic $K_3$ composition factors. This case also leads to complements which are $3$-colorable and triangle-free. Together with the work in the previous sections this allows us to prove Theorem \ref{theorem1.3} above. \\

The techniques used in this paper are greatly inspired by \cite{solvable_groups}, \cite{A5_paper}, and \cite{A5classification}. A key step in many of our arguments is studying the representation theory of $K_3$-groups and their extensions. We often depend on GAP \cite{GAP} to computationally determine the irreducible representations of a group, but where we can, we provide theoretical arguments.\\

Before we begin, we explain some notations, conventions, and important ideas that are used throughout the paper. All groups discussed in this paper are finite, and all graphs that we discuss are simple. Sets of primes will usually be denoted $\pi$, and $\pi'$ will denote the set of primes that are not in $\pi$ where the universal set is dependent on the context. We follow the convention of the ATLAS and write $G = N.M$ if $N$ is a normal subgroup of $G$ such that $G/N \cong M$. 

Graphs will usually be denoted $\Gamma$, and $\graphcomp$ will denote the complement of $\Gamma$. The notation $V(\Gamma)$ will be used to denote the vertex set of $\Gamma$, and likewise $E(\Gamma)$ will be used to denote the edge set. 

Suppose $u, v, w, x \in V(\Gamma)$ and $\{u,v\} \in E(\Gamma)$. By abuse of notation, we write $u \in \Gamma$ and $\edge{u}{v} \in \Gamma$. Also, $\digraph$ will denote an orientation of $\graphcomp$. If there are edges from $u$ to $w$, $u$ to $x$, and $w$ to $x$ in $\digraph$, we may write statements like: $u \rightarrow w \in \digraph$, $u \rightarrow w \rightarrow x \in \digraph$, $w \leftarrow u \rightarrow x \in \digraph$, and $\{u, w, x\}$ forms a triangle in $\graphcomp$. We refer to a path on $n$ edges as an $n$-path. 

If $G$ is a group, define $\pi(G)$ to be the set of prime divisors of $|G|$. The prime graph of a group $G$, denoted $\Gamma(G)$, is a simple graph constructed as follows. The set of vertices corresponds to  the set of prime divisors of $|G|$, i.e. $V(\Gamma(G))=\pi(G)$. The set of edges is formed by connecting two vertices $p$ and $q$ if and only there exists and element $g \in G$ whose order is $pq$. 

The last important idea we will need is the notion of the Frobenius digraph, first introduced in \cite{solvable_groups}. The following definitions from that paper help us.

\begin{definition}
A group $H = QP$ is called \textit{Frobenius of type $(p, q)$} if it is a Frobenius group where the Frobenius complement $P$ is a $p$-group and the Frobenius kernel $Q$ is a $q$-group. 
\end{definition}

\begin{definition}
A group $H = P_1QP_2$ is called \textit{$2$-Frobenius of type $(p, q, p)$} if it is a $2$-Frobenius group where the subgroup $P_1Q$ is Frobenius of type $(q, p)$ and the quotient group $QP_2$ is Frobenius of type $(p, q)$. 
\end{definition}

Let $G$ be a solvable group. By \cite[Theorem A]{Hall-pi-Frobenius}, if $\edge{p}{q}$ is an edge in $\graphcomp(G)$, then the Hall $\{p, q\}$-subgroup must be Frobenius or $2$-Frobenius. We can use this fact to orient the edges of $\graphcomp(G)$. 

\begin{definition}
Let $G$ be a solvable group. The \textit{Frobenius digraph} of $G$, denoted $\digraph(G)$, is the orientation of $\graphcomp(G)$ where we orient the edge $p \to q$ if the Hall $\{p, q\}$-subgroup of $G$ is Frobenius of type $(p, q)$ or $2$-Frobenius of type $(p, q, p)$. 
\end{definition}

The authors of \cite{solvable_groups} then showed that the Frobenius digraph $\digraph(G)$ of a solvable group $G$ cannot contain directed $3$-paths. This, along with the Gallai-Hasse-Roy-Vitaver theorem (\cite[Theorem 7.17]{chromatic_graph_theory}) on graph colorings implies that $\graphcomp(G)$ is $3$-colorable. We will regularly exploit this fact about Frobenius digraphs and the Gallai-Hasse-Roy-Vitaver theorem in our proofs in this paper. For example, to prove a graph is $3$-colorable, we will use the fact that it suffices to find an orientation of the graph with no directed $3$-paths. 

In general, we will not be studying solvable groups $G$, so Hall $\pi$-subgroups may not always exist. But when a Hall $\{p,q\}$-subgroup $H_{pq}$ does exist, we can use a similar analysis, considering whether $H_{pq}$ can be Frobenius or $2$-Frobenius.

\section{Preliminaries}

For the convenience of the reader, we begin by restating two lemmas from \cite{A5_paper} and \cite{3-centralizers} for later use.

\begin{lemma}\label{lem:Hallsubgroups} 
\cite[Lemma 6.4]{A5_paper} Let $H$ be a Hall $\pi$-subgroup of $G$ for some $\pi \subseteq \pi(G)$. Suppose $G$ has a normal series $G = G_1.G_2.\cdots.G_m$. Then $H$ has a normal series $H = H_1.H_2.\cdots.H_m$ where each $H_i$ is a Hall $\pi$-subgroup of $G_i$.
\end{lemma}

\begin{lemma}\label{lem:extensions_of_A6} 
\cite[Proposition 3.2]{3-centralizers}
Any group of the form $G = P.\PSL(2, q)$, where $q > 5$ is odd and $P$ is a nontrivial $p$-group for some prime $p \neq 3$, contains an element of order $3p$.
\end{lemma}

We now prove some general technical lemmas about pseudo $T$-solvable groups which we use throughout the rest of the paper.

The proof of the following lemma follows similarly to \cite[Lemma 6.2]{A5_paper}, correcting some errors made in the proof of that lemma. Indeed, since $\Aut(A_5) = S_5$, then $\pi(A_5) = \pi(\Aut(A_5))$, so this corrected proof shows that \cite[Lemma 6.2]{A5_paper} still holds.

\begin{lemma}\label{lem:subgrouplemma}
Let $G$ be strictly pseudo $T$-solvable, where $T$ is non-abelian and simple, and $\pi(\Aut(T)) = \pi(T)$. Then $G$ contains a subgroup $K \cong N.(T \times H)$,
where $N$ is solvable and $H$ is solvable with order coprime to $|T|$. Furthermore, $\pi(K) = \pi(G)$. In particular, $\graphcomp(G)$ is obtained by removing edges from $\graphcomp(K)$.
\end{lemma}

\begin{proof} Let $\pi = \pi(T)$. Since $G$ is pseudo $T$-solvable, $G$ is $\pi'$-separable. Thus, $G$ contains a Hall $\pi'$-subgroup $H$. 

Fix a chief series for $G$. Note that the chief factors of $G$ are elementary abelian or isomorphic to $T^k$ forsome positive integer $k$. First assume that the lowest chief factor is (isomorphic to) $T^k$. Since $T^k$ is normal in $G$, $H$ acts on $T^k$ by conjugation. By \cite[Corollary 3.3]{Bidwell}, $\Aut(T^k) = \Aut(T) \wr S_k$. In particular, $H$ has a permutation action on the $k$ copies of $T$. Let $X$ be one of these copies of $T$, and let $L$ be the direct product of the copies of $T$ in the orbit of $X$ in this permutation action of $H$. So $H$ acts on $L$ by conjugation. Let $B$ be the stabilizer of $X$ in the permutation action, so $B = N_H(X)$. Since $B$ is a subgroup of $H$, it has order coprime to $|\Aut(T)| = |\Aut(X)|$. Thus, $B$ must act trivially on $X$. So $C_X(B) = X$. By \cite[Lemma 2.2]{ISAACS1996125}, $C_L(H) \cong C_X(B) = X \cong T$. By considering orders, $H \cap C_L(H) = 1$. And $H$ centralizes $C_L(H)$, so $G$ has a subgroup $C_L(H) \times H \cong T \times H$. 

For an arbitrary group $G$, $T^{k}$ would not necessarily be the lowest chief factor. Let $N$ be the normal subgroup directly below $T^k$ in the chief series. Applying the above argument to $G/N$ will produce a subgroup $K^*\cong T\times H$ where $K^*\leq G/N$ and $H$ is a Hall $\pi'$-subgroup of $G/N$. The preimage of $K^*$ under the canonical homomorphism $G\to G/N$ is a subgroup $K$ of $G$ isomorphic to $N.(T\times H)$.

Note that $\pi(G/N) = \pi(T \times H)$ by construction of $H$. Thus, $\pi(K) = \pi(G)$, which implies $\graphcomp(G)$ is obtained by removing edges from $\graphcomp(K)$.
\end{proof}

The solvable group $H$ from Lemma \ref{lem:subgrouplemma} plays no important role in our later analysis, so we provide the following corollary to simplify the result. We remark that $H$ plays no important role in \cite{A5_paper} either, so a similar reduction is possible in that paper as well.

\begin{corollary}\label{cor:subgroupcor}
Let $G$ be strictly pseudo $T$-solvable, where $T$ is nonabelian and simple, and $\pi(\Aut(T)) = \pi(T)$. Then $G$ contains a subgroup $K \cong N.T$,
where $N$ is solvable and $\pi(K) = \pi(G)$. In particular, $\graphcomp(G)$ is obtained by removing edges from $\graphcomp(K)$.
\end{corollary}

\begin{proof}
By Lemma \ref{lem:subgrouplemma}, $G$ has a subgroup $K \cong N_1.(T \times H)$, where $N_1$ and $H$ are solvable, and $\pi(K) = \pi(G)$. Thus, $K = N_1.H.T$, and setting $N = N_1.H$, we obtain $K = N.T$ where $N$ is solvable.
\end{proof}



\begin{lemma}\label{lem:onecopyofT}
Suppose $T$ is a non-abelian simple group and $p$ and $q$ are distinct primes dividing the order of $T$. If $\edge{p}{q}$ is an edge in $\graphcomp(G)$, then there is at most one copy of $T$ in the composition series of $G$.
\end{lemma}

\begin{proof}
If $G$ has no composition factors isomorphic to $T$, we are done, so suppose $G$ has a composition factor isomorphic to $T$. Take any chief series of $G$. In this chief series, every chief factor must be isomorphic to the direct product of several copies of the same simple group. We claim that if a chief factor is isomorphic to $T^k$ then $k = 1$. If $k > 1$, then $T^k$ has an element of order $pq$, and so $G$ does as well. This contradicts that fact that $\edge{p}{q}$ is an edge in $\graphcomp(G)$, and so $k = 1$. 

We now consider whether $T$ can be a Frattini chief factor. Since the Frattini subgroup of a finite group is nilpotent, any Frattini chief factor is abelian. Therefore, each chief factor isomorphic to $T$ is a non-Frattini chief factor. Since these factors are non-Frattini, then by \cite[Theorem A]{ballester-bolinches}, there exist normal subgroups $C, R \unlhd G$ such that $C \leq R$ and $R / C \cong T^n$. In particular, $n$ is the number of (non-Frattini) chief factors isomorphic to $T$. 

If $n > 1$, then there exists an element of order $pq$ in $T^n$ and therefore in $R \leq G$, a contradiction. However, $G$ has at least one composition factor isomorphic to $T$. Then $n = 1$ and so $G$ has exactly one chief factor isomorphic to $T$. Thus, $G$ must have exactly one composition factor isomorphic to $T$.
\end{proof}

The following lemma generalizes \cite[Theorem 6.6]{A5_paper}, and the proof follows a similar argument. 

\begin{lemma}\label{lem:remove_p-r_then_solvable}
Let $G$ be a pseudo $T$-solvable group. Let $T$ be a simple non-abelian group with $|\pi(T)|=3$ (i.e., $T$ is a $K_3$-group). Call these prime divisors $p$, $q$, and $r$. If $T$ contains subgroups $S_{pq}$ and $S_{qr}$ such that $\pi(S_{pq}) = \{p,q\}$ and $\pi(S_{qr}) = \{q,r\}$, then $\graphcomp(G) \setminus \{\edge{p}{r}\}$ is $3$-colorable and triangle-free.
\end{lemma}

\begin{proof}
Note that if $G$ is solvable, then $\graphcomp(G)$ is $3$-colorable and triangle-free by Lemma \ref{lem:solvablegroups}. So we may assume $G$ is strictly pseudo $T$-solvable. 
Now all eight $K_3$-groups $T$ fulfill the condition that $\pi(\Aut(T)) = \pi(T)$. Thus Corollary \ref{cor:subgroupcor} can be applied to $G$ to find a subgroup $K \cong N.T$, where $N$ is solvable and $\pi(G) = \pi(K)$. Let $K_1 = N.S_{pq}$ and $K_2 = N.S_{qr}$ be subgroups of $K$. Notice that $S_{pq}$ and $S_{qr}$ are solvable by Burnside's Theorem. Thus $K_1$ and $K_2$ are solvable, and $\graphcomp(K_1)$ and $\graphcomp(K_2)$ are $3$-colorable and triangle-free by Lemma \ref{lem:solvablegroups}. If $r \mid |N|$, then $\pi(K_1) = \pi(G)$ and so $\graphcomp(G)$ is a subgraph of $\graphcomp(K_1)$. Therefore, $\graphcomp(G)$ is $3$-colorable and triangle free because these properties are closed under removing edges. Similarly if $p \mid |N|$ then $\graphcomp(G)$ is a subgraph of $\graphcomp(K_2)$ and thus $3$-colorable and triangle free.

Now assume $p,r \nmid |N|$. Then $\pi(K_1)=\pi(K) \setminus \{r\}$ and $\pi(K_2)=\pi(K) \setminus \{p\}$. Therefore any triangle in $\graphcomp(K)$ must include both the $r$ and $p$ vertices. Thus removing the edge $\edge{p}{r}$ will remove any triangles in $\graphcomp(K)$ and by extension in $\graphcomp(G)$. Next we will induce an orientation on $\graphcomp(K)$ and show that there are no $3$-paths. Assign the orientation to $\graphcomp(K)$ by taking the orientations of the Frobenius digraphs $\digraph(K_1)$ and/or $\digraph(K_2)$. We claim this orientation is well-defined. Let $a, b \in \pi(K)$, and suppose $\edge{a}{b} \in \graphcomp(K)$. It suffices to show that the orientations of $\edge{a}{b}$ in $\digraph(K_1)$ and $\digraph(K_2)$ coincide. There are several cases.
\begin{enumerate}
    \item $a,b \notin \{p,q,r\}$. In this case, the Hall $\{a,b\}$-subgroups of $K_1$ and $K_2$ are both isomorphic to the Hall $\{a,b\}$-subgroup of $N$, so their Frobenius actions coincide. Thus the orientation is the same in $\digraph(K_1)$ and $\digraph(K_2)$, so the orientation of the edge $\edge{a}{b}$ in $\graphcomp(K)$, if it exists, is well-defined.
    \item $a=q$ and $b \notin \{p,q,r\}$. By Lemma \ref{lem:Hallsubgroups}, the Hall $\{q,b\}$-subgroup of $K_1$ (or respectively $K_2$) is of the form $A.Q$ where $A$ is the Hall $\{q,b\}$-subgroup of $N$ and $Q$ is the Sylow $q$-subgroup of $S_{pq}$ (or respectively $S_{qr}$). Since this Hall subgroup has a quotient which is a $q$-group, it must be Frobenius of type $(q,b)$ or $2$-Frobenius of type $(q,b,q)$. Therefore the orientation must be $q \to b$.
    \item $a=p$ and $b \notin \{p,q,r\}$. In this case, the edge $\edge{p}{b}$ is only in $\digraph(K_1)$ and thus the orientation is unambiguously defined. By Lemma \ref{lem:Hallsubgroups}, the Hall $\{p,b\}$-subgroup of $K_1$ is $B.P$ where $B$ is the Sylow $b$-subgroup of $N$ and $P$ is the Sylow $p$-subgroup of $S_{pq}$. By the argument in (2), the orientation must be $p \to b$.
    \item $a=r$ and $b \notin \{p,q,r\}$. In this case, the edge $\edge{r}{b}$ is only in $\digraph(K_2)$ and thus the orientation is unambiguously defined. By Lemma \ref{lem:Hallsubgroups}, the Hall $\{r,b\}$-subgroup of $K_2$ is $B.R$ where $B$ is the Sylow $b$-subgroup of $N$ and $R$ is the Sylow $p$-subgroup of $K_2$. By the argument in (2), the orientation must be $r \to b$.
    \item $a = p$ and $b = q$. In this case, the edge $\edge{p}{q}$ is only in $\digraph(K_1)$, so the orientation is unambiguous. The exact orientation depends on the structure of $S_{pq}$.
    \item $a = q$ and $b = r$. In this case, the edge $\edge{q}{r}$ is only in $\digraph(K_2)$, so the orientation is unambiguous. The exact orientation depends on the structure of $S_{qr}$.
\end{enumerate}
Therefore, this orientation of $\graphcomp(K) \setminus \{\edge{p}{r}\}$, which we denote $\digraph(K)$, is well-defined. 

Notice that any directed $3$-path must involve both $p$ and $r$. Otherwise, we would find a $3$-path in either $\digraph(K_1)$ or $\digraph(K_2)$. Examine the edges between the vertices $p, q, r$. (Below 3-path will mean ``directed 3-path''.)
    \begin{enumerate}
        \item If the orientation is $p \leftarrow q \rightarrow r$, then the only paths directed into $p$ and $r$ come from $q$, so any $3$-path would need to contain both $q \to p$ and $q \to r$. But since $q$ is a source, both segments cannot be in the same $3$-path. Therefore, there are no $3$-paths.
        \item Suppose the orientation is $p \rightarrow q \leftarrow r$, or only one or none of these edges exist. Then $p$ and $r$ are both sources and so they both must initiate any path they are a part of. Thus no $3$-path can contain both of them, so there are no $3$-paths.
        \item The two cases $p \leftarrow q \leftarrow r$ and $p \rightarrow q \rightarrow r$ are the same up to symmetry and thus we will focus exclusively on the latter. Any $3$-path in this case would need to be of the form $p \to q \to r \to a$. In this case, flip the $\edge{q}{r}$ edge to make it $q \leftarrow r$. This removes the $3$-path. The shrewd reader may object that this could induce a new $3$-path of the form $r \to q \to b \to c$. However, this would imply that $p \to q \to b \to c \in \digraph(K_1)$, but $\digraph(K_1)$ has no $3$-paths. Thus no $3$-paths can exist at all after flipping the $\edge{q}{r}$ orientation. 
        \item If $\graphcomp(K)$ contains only one or none of the edges $p-q-r$, then $\graphcomp(K)$ is a subgraph of case (1) or (2). Since neither of those cases have $3$-paths, then $\graphcomp(K)$ cannot have $3$-paths either.
    \end{enumerate}
Finally we see that in all possible scenarios, $\digraph(K)$ has no $3$-paths. Thus $\graphcomp(K) \setminus \{\edge{p}{r}\}$ is $3$-colorable, and so is $\graphcomp(G) \setminus \{\edge{p}{r}\}$.
\end{proof}

As Table \ref{table:1} in the Appendix below in Section 10 illustrates, all of the $K_3$-groups satisfy the assumptions of Corollary \ref{cor:subgroupcor} and Lemmas \ref{lem:onecopyofT} and \ref{lem:remove_p-r_then_solvable}. In particular, Lemma \ref{lem:remove_p-r_then_solvable} provides a necessary condition for a graph to be the prime graph of a pseudo $T$-solvable group, where $T$ is a $K_3$-group.

The following representation theory lemma will also be useful.

\begin{lemma}\label{lem:rep_theory_biz}
Let $G = N.T$ where $N$ and $T$ are finite groups and $(|N|, |T|) = 1$. Let $p \in \pi(T)$ and $q \in \pi(N)$. If there exists an element $t \in T$ of order $p$ such that for every irreducible representation $\rho: T \to \GL(n, \C)$, $\rho(t)$ has a fixed point, then there exists an element of order $pq$ in $G$.
\end{lemma}

\begin{proof}
Since $(|N|, |T|) = 1$, by the Schur-Zassenhaus theorem, the extension $G = N.T$ splits, so $T$ acts on $N$. In order to find an element of order $pq$ in $G$, it suffices to find an element of order $p$ in $T$ which fixes an element of order $q$ in $N$.

As a property of coprime actions, there is a Sylow $q$-subgroup $Q \leq N$ which is invariant under the $T$-action. So we may consider the subgroup $Q \rtimes T \leq G$. We may apply the Hartley-Turull lemma (\cite[Lemma 2.6.2]{Hartley1994}) to see that $T$ acts on an elementary abelian $q$-group $V$ such that the action of $T$ on $Q$ is equivalent to the action of $T$ on $V$. So it suffices to show that some order $p$ element of $T$ fixes a nontrivial element of $V$.

Since $V$ is elementary abelian, we may view it as an $\F_qT$-module. By Maschke's theorem, we may decompose $V$ into the direct sum of irreducible $\F_qT$-modules. Let $W$ be one such direct summand. If $K$ is the kernel of the action of $T$ on $W$, then $T / K$ acts faithfully and irreducibly on $W$. Let $t \in T$ be the order $p$ element of $T$ such that for every irreducible $\rho: T \to \GL(n, \C)$, $\rho(t)$ has a fixed point. If $t \in K$, then we immediately retrieve an element of order $pq$ and we are done. So assume $t \notin K$. If $\phi: T \to T/K$ is the natural projection with kernel $K$, then $\phi(t)$ has order $p$. Also, since $T/K$ is a quotient of $T$, then for every irreducible representation $\sigma: T/K \to \GL(m, \C)$, $\sigma(\phi(t))$ acts with a fixed point. So it suffices to show that $W$ has an order $q$ element that is fixed by $\phi(t)$ in the action of $T/K$ on $W$.

Abbreviate $S = T/K$ and $s = \phi(t)$. Since $W$ is a faithful irreducible $\F_qS$-module, we may apply \cite[Lemma 10]{ROBINSON19961143}, which states that if $k = \on{End}_{\F_q S}(W)$, then the irreducible summands of $W \otimes k$ are absolutely irreducible, and that the permutation action of $S$ on any such summand is equivalent to the permutation action of $S$ on $W$. Let $U$ be one such summand. Since the action of $S$ on $U$ is coprime, and $U$ is absolutely irreducible, it is well known that $U$ may be viewed as a module over the complex numbers. (For example, this fact can be extracted from \cite{RepTheoryBook}.) 

Since $s$ acts with a fixed point under every irreducible complex representation of $S$, $s$ must act on $U$ with a fixed point. As discussed above, this entails that $s = \phi(t)$ must act on $W$ with a fixed point as well. Then, $t$ acts on $W$ with a fixed point, and as discussed above, thus $T$ acts on $Q$ with a fixed point. Thus, we retrieve an element of order $pq$ in $G$ and we are done.
\end{proof}

The general trajectory of the rest of the paper is to address each of the $K_3$-groups $T$ (aside from $A_5$) and determine necessary and sufficient conditions for when a graph $\Gamma$ is the prime graph of a pseudo $T$-solvable group. Each $K_3$-group has its own section, aside from $\PSL(2, 7)$ and $U_3(3)$, which we address together, and $A_6$ and $U_4(2)$, which we also address together.

\section{The Projective Special Linear Group $\PSL(2,7)$ and the Unitary Group $U_3(3)$}\label{section:psl27}

In this section, we fully classify the prime graphs of pseudo $\PSL(2,7)$-solvable groups, adapting the methodology of \cite{A5_paper} and \cite{A5classification}. First we establish that for an arbitrary pseudo $\PSL(2,7)$-solvable group $G$, removing the $\edge{2}{7}$ edge from $\graphcomp(G)$ yields a 3-colorable, triangle-free graph. We then work toward showing that the only possible triangle in $\graphcomp(G)$ is $\{2,3,7\}$. We show that when this triangle occurs, $2$ and $3$ cannot be connected to any other primes in $\graphcomp(G)$, while $7$ can connect to other vertices. The final result is that $\graphcomp(G)$ must be $3$-colorable and contain either no triangles or exactly one $\{2,3,7\}$ triangle where $7$ is the only vertex which can connect to the rest of the graph, and in this case, a valid $3$-coloring exists where all the vertices which $7$ connects to, except $2$ and $3$, are of the same color. We also show that given such a graph, one can construct a pseudo $\PSL(2,7)$-solvable group such that $\graphcomp(G)$ is isomorphic to the given graph by following a similar method to that used in \cite[Theorem 2.8]{solvable_groups} and \cite[Theorem 3.1]{A5classification}.

We also find as an immediate result of our classification that prime graph complements of pseudo $U_3(3)$-solvable groups must be $3$-colorable and triangle-free. We leverage the fact that $\PSL(2, 7) \leq U_3(3)$, as well as the fact that $U_3(3)$ has an element of order $6$.

As a first step in our classification, we find a necessary condition for a graph to be the prime graph of a pseudo $\PSL(2,7)$-solvable group. 

\begin{theorem}\label{thm:psuedopsl27primegraphs}
Let $G$ be pseudo $\PSL(2,7)$-solvable. Then $\graphcomp(G) \setminus \{\edge{2}{7}\}$ is $3$-colorable and triangle-free.
\end{theorem}

\begin{proof}
Let $p = 2, q = 3, r = 7$. Apply Lemma \ref{lem:remove_p-r_then_solvable} with the subgroups $A_4, F_{21} \leq \PSL(2, 7)$, where $F_{21}$ is the Frobenius group of order $21$.
\end{proof}

From this proof we can define the Frobenius digraph of a pseudo $\PSL(2,7)$-solvable group. 

\begin{definition}
If $G$ is a pseudo $\PSL(2,7)$-solvable group, we say its \textit{Frobenius digraph} $\digraph(G)$ is the orientation of $\graphcomp(G)$ given by the orientation of $\digraph(K)$ as defined in the proof of Theorem \ref{thm:psuedopsl27primegraphs} (which is provided in the proof of Lemma \ref{lem:remove_p-r_then_solvable}). In addition, if there is an edge $\edge{2}{7}$ in $\graphcomp(G)$, we direct it by $7 \rightarrow 2$. 
\end{definition}

By the proof of Lemma \ref{lem:remove_p-r_then_solvable}, if the edges $\edge{2}{3}$ and/or $\edge{3}{7}$ are in $\graphcomp(G)$, the orientations must be $2 \leftarrow 3 \rightarrow 7$. The choice of $7 \rightarrow 2$ over $2 \rightarrow 7$ is somewhat arbitrary, but it helps for consistency in the proof of Theorem \ref{thm:PSL_classification}. 

Our next step is to investigate the conditions under which such triangles might exist. Theorem \ref{thm:psuedopsl27primegraphs} tells us that if $\graphcomp(G)$ has a triangle, it must involve the edge $\edge{2}{7}$. Since the prime graph of $\PSL(2, 7)$ itself is complete, it is clear that a $\{2, 3, 7\}$ triangle can exist. In fact, we will show that for any prime $p \neq 3$, a $\{2, 7, p\}$ triangle cannot exist.

\begin{proposition}\label{prop:2-p}
Let $G$ be a strictly pseudo $\PSL(2,7)$-solvable group, and let $p$ be a prime such that $p \notin \{2,3,7\}$. If $\edge{2}{p}$ is an edge in $\graphcomp(G)$, then $\graphcomp(G)$ is $3$-colorable and triangle-free.
\end{proposition}

\begin{proof}
Suppose that $\edge{2}{p}$ is an edge in $\graphcomp(G)$. By Corollary \ref{cor:subgroupcor}, we have a subgroup $K = N.\PSL(2,7)$ of $G$, with $N$ solvable, such that $\pi(G) = \pi(K)$. We claim $2 \mid |N|$.

Suppose, on the contrary, that $2 \nmid |N|$. Note that since the subgroup $N.D_8$ of $K$ is solvable, it has a Hall $\{2,p\}$-subgroup $H_{2p}$. By Lemma \ref{lem:Hallsubgroups}, $H_{2p} = P.D_8 = P \rtimes D_8$, where $P$ is a Sylow $p$-subgroup of $N$. Since $D_8$ cannot act Frobeniusly, $H_{2p}$, and thus $G$, contains an element of order $2p$. But $\edge{2}{p} \in \graphcomp(G)$, a contradiction. 

Therefore, $2 \mid |N|$. Taking the solvable subgroup $K_1 = N.F_{21}$ of $K$, we know $\graphcomp(K_1)$ is $3$-colorable and triangle-free. Since $2 \mid |N|$, $\pi(K_1) = \pi(K) = \pi(G)$. This implies $\graphcomp(G)$ is a subgraph of $\graphcomp(K_1)$, so $\graphcomp(G)$ is $3$-colorable and triangle-free.
\end{proof}

We immediately obtain the following corollaries to Proposition \ref{prop:2-p}.

\begin{corollary}\label{No2-7-p}
If $G$ is a pseudo $\PSL(2,7)$-solvable group and $p \notin \{2,3,7\}$, then there is no $\{2, 7, p\}$ triangle in $\graphcomp(G)$.
\end{corollary}

\begin{corollary}\label{isolated 2}
If $G$ is a pseudo $\PSL(2,7)$-solvable group and $\{2,3,7\}$ forms a triangle in $\graphcomp(G)$, then the vertex $2$ is not connected to any other vertices in $\graphcomp(G)$.
\end{corollary}

The above arguments show that if $G$ is a pseudo $\PSL(2, 7)$-solvable group, then at most one triangle exists in $\graphcomp(G)$, and it must be the $\{2, 3, 7\}$ triangle. If this triangle does not exist, then $\graphcomp(G)$ is triangle-free and $3$-colorable, as we will show in Theorem \ref{thm:PSL_classification}. This gives us the constraints as for solvable groups. However, if such a triangle does exist, then there are still multiple questions to ask, such as classifying the outgoing edges from $2$, $3$, and $7$. By Corollary \ref{isolated 2}, we already have that there are no outgoing edges from $2$ besides $\edge{2}{3}$ and $\edge{2}{7}$. We will prove that the same holds for $3$, but not for $7$. In particular, it turns out that $7$ may have an arbitrary number of outgoing edges.


\begin{proposition}\label{prop:no3pedges}
If $G$ is a pseudo $\PSL(2,7)$-solvable group and $\{2,3,7\}$ forms a triangle in $\graphcomp(G)$, then the vertex $3$ is not connected to any other vertices in $\graphcomp(G)$.
\end{proposition}

\begin{proof}
Suppose for contradiction that the edge $\edge{3}{p}$ is in $\graphcomp(G)$ for some prime $p$ coprime to $42$. Since $\{2, 3, 7\}$ forms a triangle in $\graphcomp(G)$, $G$ is not solvable. Thus, by Corollary \ref{cor:subgroupcor}, $G$ has a subgroup $K = N.\PSL(2,7)$, where $N$ is solvable and $\pi(K) = \pi(G)$. We claim $3 \nmid |N|$.

Consider the solvable subgroup $K_1 = N.S_4$ of $K$, and take a Hall $\{2,3\}$-subgroup $H_{23}$ of $K_1$. By Lemma \ref{lem:Hallsubgroups}, $H_{23} = A.S_4$, where $A$ is a Hall $\{2, 3\}$-subgroup of $N$. Since $\edge{2}{3}$ is an edge in $\graphcomp(G)$, it must be an edge in $\graphcomp(H_{23})$. Therefore, $H_{23}$ is Frobenius or $2$-Frobenius. But $S_4$ sits at the top of the normal series for $H_{23}$, so $H_{23}$ must be $2$-Frobenius of type $(2,3,2)$. In particular, $3 \nmid |A|$, so $3 \nmid |N|$. 

Now let $\pi = \{2, 3, 7, p\}$. Since the only composition factors of $K$ are either cyclic or $\PSL(2, 7)$, $K$ is $\pi$-separable, so it has a Hall $\pi$-subgroup $H_\pi$. By Lemma \ref{lem:Hallsubgroups}, $H_\pi = B.\PSL(2, 7)$, where $B$ is a Hall $\{2,7,p\}$-subgroup of $N$. 

Take some order $3$ element $t \in H_\pi \setminus B$, and consider $C_B(t)$. If $C_B(t)$ is nontrivial, then $t$ commutes with an element of order $2$, $7$, or $p$. But $\edge{2}{3}, \edge{3}{7}, \edge{3}{p} \in \graphcomp(K)$, so there are no elements of order $6, 21$, or $3p$ in $H_\pi$, a contradiction. Thus, $C_B(t)$ must be trivial. By Lemma \ref{lem:extensions_of_A6}, this means $B = 1$. However, $p \mid |B|$, a contradiction. This completes the proof.
\end{proof}

Finally, we have our main classification theorem for prime graphs of pseudo $\PSL(2,7)$-solvable groups. The construction in the following theorem closely follows the construction from \cite[Theorem 3.1]{A5classification} for $\PSL(2,7)$ rather than $A_5$.

\begin{theorem}\label{thm:PSL_classification}
Let $\Gamma$ be a simple graph. Then $\Gamma$ is isomorphic to the prime graph of a pseudo $\PSL(2, 7)$-solvable group $G$ if and only if one of the following holds:
\begin{enumerate}
    \item $\graphcomp$ is triangle-free and $3$-colorable.
    \item $\graphcomp$ has one exactly one triangle $\{a,b,c\}$, the vertices $a, b$ are not connected to any other vertices in $\graphcomp$, and $\graphcomp$ has a $3$-coloring for which all the neighbors of $c$ other than $a$ and $b$ have the same color.
\end{enumerate}
\end{theorem}

\begin{proof}
We first prove the forward direction. Assume that $\Gamma$ is the prime graph of a pseudo $\PSL(2,7)$-solvable group $G$. If $\graphcomp$ does not contain the edge $\edge{2}{7}$, then $\graphcomp$ is $3$-colorable and triangle-free by Theorem \ref{thm:psuedopsl27primegraphs}. Now suppose $\graphcomp(G)$ has a $\edge{2}{7}$ edge. Since removing the vertex $2$ from $\graphcomp(G)$ removes the edge $\edge{2}{7}$, $\graphcomp(G) \setminus \{2\}$ is $3$-colorable by Theorem \ref{thm:psuedopsl27primegraphs}. If $2$ is connected to some prime $p \notin \{2,3,7\}$, then by Proposition \ref{prop:2-p}, $\graphcomp(G)$ is $3$-colorable. Otherwise, the vertex $2$ can have at most degree $2$ (with one edge to $3$ and one to $7$). Therefore adding the vertex $2$ back to $\graphcomp(G) \setminus \{2\}$ will not violate $3$-colorability. Thus the complement of the prime graph of any pseudo $\PSL(2,7)$-solvable group will be $3$-colorable.

Now we consider the triangles in $\graphcomp(G)$. If there is no triangle in $\graphcomp(G)$ we are done, so suppose that there is a triangle. It must contain the edge $\edge{2}{7}$ due to Theorem \ref{thm:psuedopsl27primegraphs}. Therefore it must be a $\{2,3,7\}$ triangle by Corollary \ref{No2-7-p}, and it must be the only triangle in $\graphcomp(G)$. Furthermore, the vertices $2$ and $3$ are disconnected from all other primes by Corollary \ref{isolated 2} and Proposition \ref{prop:no3pedges}. Now examine the Frobenius Digraph $\digraph(G)$. Note that there will be no $3$-paths in this case by the proof of Theorem \ref{thm:psuedopsl27primegraphs}. If $p \notin \{2,3,7\}$ is an arbitrary prime adjacent to $7$, we have that the orientation is $3 \to 7 \to p$. Since there are no $3$-paths in this digraph, if $q$ is a prime adjacent to $p$, then the orientation of the $\edge{p}{q}$ edge must be $q \to p$. Now define a coloring on this digraph. Label all vertices with zero out-degree with $\mathcal{I}$, all vertices with zero in-degree and nonzero out-degree with $\mathcal{O}$, and all vertices with non-zero in- and out-degree with $\mathcal{D}$. This is a well-defined $3$-coloring, since $\digraph(G)$ contains no $3$-paths. And all vertices adjacent to $7$ other than $2$ and $3$ are labeled with $\mathcal{I}$. This concludes the forward direction of the proof. 

We now prove the backward direction. In each case, we will find a group $G$ whose prime graph is isomorphic to $\Gamma$. If we have a graph $\Gamma$ which satisfies (1), then by \cite[Theorem 2.8]{solvable_groups}, there exists a solvable group $G$ for which $\Gamma \cong \Gamma(G)$. Since $G$ is trivially pseudo $\PSL(2,7)$-solvable, this case is complete. 

Suppose we have a graph $\Gamma$ which satisfies (2). If the triangle $\{a,b,c\}$ is isolated in $\graphcomp$, then by the construction in \cite[Theorem 2.8]{solvable_groups}, there exists a solvable group $H$ with a prime graph whose complement is isomorphic to $\graphcomp \setminus \{a,b,c\}$ such that $2, 3, 7 \notin V(\Gamma(H))$. Then setting $G = \PSL(2,7) \times H$, we are done. So we may assume $c$ is connected in $\graphcomp$ to at least one vertex other than $a$ or $b$. 

Take the $3$-coloring $\{\mathcal{O}, \mathcal{D}, \mathcal{I}\}$ of $\graphcomp$ where $\mathcal{D}$ is the color containing the vertex $c$, $\mathcal{I}$ is the color containing the neighbors of $c$ other than $a$ and $b$, and $\mathcal{O}$ is the third color. We direct the edges of $\graphcomp$ by assigning the arrows pointing from $\mathcal{O} \rightarrow \mathcal{D}$, $\mathcal{O} \rightarrow \mathcal{I}$, and $\mathcal{D} \rightarrow \mathcal{I}$. Without loss of generality, suppose $a \in \mathcal{I}$ and $b \in \mathcal{O}$. Let $\digraph$ denote this orientation. Also, let $n_o = |\mathcal{O} \backslash \{b\}|$, $n_d = |\mathcal{D} \backslash \{c\}|$, and $n_i = |\mathcal{I} \backslash \{a\}|$. 

For the reader's convenience, we have provided Figure \ref{fig:psl_construction} to help visualize the construction. Choose $n_o$ distinct primes $p_1, \dots, p_{n_o}$ other than $2, 3, 7$. Set $p$ to be the product of these primes. By Dirichlet's theorem on arithmetic progressions, then pick a set of distinct primes $q_1, \dots, q_{n_d}$ other than $2, 3, 7$, such that $q_i \equiv 1$ (mod $p$) for all $i$. Identify each vertex in $\mathcal{O} \setminus \{b\}$ with one of the $p_j$, and identify each vertex in $\mathcal{D} \setminus \{c\}$ with one of the $q_i$. Define groups $P = C_{p_1} \times \dots \times C_{p_{n_o}}$, $Q = C_{q_1} \times \dots \times C_{q_{n_d}}$. For all indices $i, j$, if $p_j \rightarrow q_i$ is an edge in $\digraph$, then let $C_{p_j}$ act Frobeniusly on $C_{q_i}$. This is possible because $q_i \equiv 1$ (mod $p_i$). Otherwise, if $p_j$ and $q_i$ are not adjacent in $\graphcomp$, let $C_{p_j}$ act trivially on $C_{q_k}$. This defines a group action of $P$ on $Q$, so we obtain the induced semidirect product $K = Q \rtimes P$. Note that $K$ is solvable. By this construction, $\digraph(K) \cong \digraph[(\mathcal{O} \setminus \{b\}) \cup (\mathcal{D} \setminus \{c\})]$. Then, identifying the vertex $a$ with $2$, $b$ with $3$, and $c$ with $7$, we get that the Frobenius digraph of the pseudo-$\PSL(2,7)$ solvable group $\PSL(2,7) \times K$ is isomorphic to $\digraph[\mathcal{O} \cup \mathcal{D} \cup \{a\}]$. 

Now, let $v_1, \dots, v_{n_i}$ be the vertices in $\mathcal{I} \setminus \{a\}$. We will eventually identify each $v_k$ with another prime $r_k$. For each index $k$, let $N^1(v_k)$ and $N^2(v_k)$ denote the sets of primes in $\graphcomp \setminus \{a,b,c\}$ with in-distance $1$ and $2$ to $v_k$, respectively. If $N^1(v_k)$ is nonempty, let $B_k$ be a Hall $(N^1(v_k) \cup N^2(v_k))$-subgroup of $K$. By definition of $K$, $\Fit(B_k)$ is a Hall $N^1(v_k)$-subgroup of $K$. If $N^1(v_k)$ is empty, set $B_k = 1$. Now, there are two cases:
\begin{enumerate}
    \item \ul{$7 \rightarrow v_k$ is an edge in $\digraph$}. Take the irreducible complex representation $\rho_1$ of $\PSL(2,7)$ listed in Table \ref{table:3} for which elements of order $7$ act fixed-point freely and all other elements have fixed points. Then by Dirichlet's theorem on arithmetic progressions, pick a prime $r_k$ such that $|\PSL(2,7) \times B_k| \mid (r_k - 1)$. By \cite[Lemma 3.5]{A5classification}, there exists a modular representation $R_k$ of $\PSL(2,7) \times B_k$ over a finite field of characteristic $r_k$ such that $\Fit(B_k)$ acts Frobeniusly on $R_k$. Also, elements of order $2$ and $3$ have a fixed point in $R_k$, while elements of order $7$ act Frobeniusly on $R_k$.
    
    \item \ul{$7,v_k$ are not adjacent in $\graphcomp$}. Consider the trivial complex representation of $\PSL(2,7)$, and pick a prime $r_k$ such that $|\PSL(2,7) \times B_k| \mid (r_k - 1)$. By \cite[Lemma 3.5]{A5classification}, there then exists a modular representation $R_k$ of $\PSL(2,7) \times B_k$ over a finite field of characteristic $r_k$ such that $\Fit(B_k)$ acts Frobeniusly on $R_k$. Also, $\PSL(2,7)$ acts trivially on $R_k$.
\end{enumerate}
Let $J = R_1 \times \dots \times R_{n_i}$. Note that by Dirichlet's theorem on arithmetic progressions, we can require each of the $r_k$ to be distinct from the $p_j$ and $q_i$, and we can require the $r_k$ to be mutually distinct. Thus, we have defined an action of each $\PSL(2,7) \times B_k$ on $R_k$. This induces an action of $\PSL(2,7) \times K$ on $J$, which induces the pseudo $\PSL(2,7)$-solvable semidirect product $G = J \rtimes (\PSL(2,7) \times K)$. By this construction, $\digraph(G) \cong \digraph$. Therefore, $\Gamma(G) \cong \Gamma$. 
\end{proof}

\begin{figure}
    \centering

\begin{tikzpicture}
\node (7) at (0, 0) {$\bullet$};
\node (2) at (-1, 1.73) {$\bullet$};
\node (3) at (1, 1.73) {$\bullet$};

\node (r1) at (-5, -2) {$\bullet$};
\node (r2) at (-3, -2) {$\bullet$};
\node (r3) at (-1, -2) {$\bullet$};
\node (r4) at (1, -2) {$\bullet$};
\node (r5) at (3, -2) {$\cdots$};
\node (r6) at (5, -2) {$\bullet$};

\node (q1) at (-5, -4) {$\bullet$};
\node (q2) at (-3, -4) {$\bullet$};
\node (q3) at (-1, -4) {$\bullet$};
\node (q4) at (1, -4) {$\bullet$};
\node (q5) at (3, -4) {$\cdots$};
\node (q6) at (5, -4) {$\bullet$};

\node (p1) at (-5, -6) {$\bullet$};
\node (p2) at (-3, -6) {$\bullet$};
\node (p3) at (-1, -6) {$\bullet$};
\node (p4) at (1, -6) {$\bullet$};
\node (p5) at (3, -6) {$\cdots$};
\node (p6) at (5, -6) {$\bullet$};

\draw (7) -- (3);
\draw (7) -- (2);
\draw (3) -- (2);

\draw (7) -- (r1);
\draw (7) -- (r4);
\draw (7) -- (r6);

\draw (r2) -- (q2);
\draw (r2) -- (p3);
\draw (r1) -- (q2);
\draw (q4) -- (p1);
\draw (q2) -- (p1);
\draw (q4) -- (p6);
\draw (r6) -- (q4);
\draw (r6) -- (q6);
\draw (q3) -- (p4);
\draw (q4) -- (p3);
\draw (r6) -- (q2);
\draw (r1) -- (p2);

\draw[fill=red!20!white] (r1) circle (0.4);
\draw[fill=red!20!white] (r2) circle (0.4);
\draw[fill=red!20!white] (r3) circle (0.4);
\draw[fill=red!20!white] (r4) circle (0.4);
\draw[fill=red!20!white] (r6) circle (0.4);

\draw[fill=green!20!white] (q1) circle (0.4);
\draw[fill=green!20!white] (q2) circle (0.4);
\draw[fill=green!20!white] (q3) circle (0.4);
\draw[fill=green!20!white] (q4) circle (0.4);
\draw[fill=green!20!white] (q6) circle (0.4);

\draw[fill=blue!20!white] (p1) circle (0.4);
\draw[fill=blue!20!white] (p2) circle (0.4);
\draw[fill=blue!20!white] (p3) circle (0.4);
\draw[fill=blue!20!white] (p4) circle (0.4);
\draw[fill=blue!20!white] (p6) circle (0.4);

\draw[fill=green!20!white] (7) circle (0.4);
\draw[fill=blue!20!white] (3) circle (0.4);
\draw[fill=red!20!white] (2) circle (0.4);

\node at (2) {$2$};
\node at (3) {$3$};
\node at (7) {$7$};

\node at (r1) {$r_1$};
\node at (r2) {$r_2$};
\node at (r3) {$r_3$};
\node at (r4) {$r_4$};
\node at (r6) {$r_\ell$};

\node at (q1) {$q_1$};
\node at (q2) {$q_2$};
\node at (q3) {$q_3$};
\node at (q4) {$q_4$};
\node at (q6) {$q_m$};

\node at (p1) {$p_1$};
\node at (p2) {$p_2$};
\node at (p3) {$p_3$};
\node at (p4) {$p_4$};
\node at (p6) {$p_n$};

\node at (6.25, -2) {$\cal{I}$};
\node at (6.25, -4) {$\cal{D}$};
\node at (6.25, -6) {$\cal{O}$};
\node at (-6.25, -2) {$\;$};

\end{tikzpicture}
    \caption{This graph $\graphcomp$ is the complement to a graph $\Gamma$ which satisfies condition (2) of Theorem \ref{thm:PSL_classification}, along with the given $3$-coloring. The vertices are labeled with the corresponding primes which are assigned during the construction in the theorem.}
    \label{fig:psl_construction}
\end{figure}
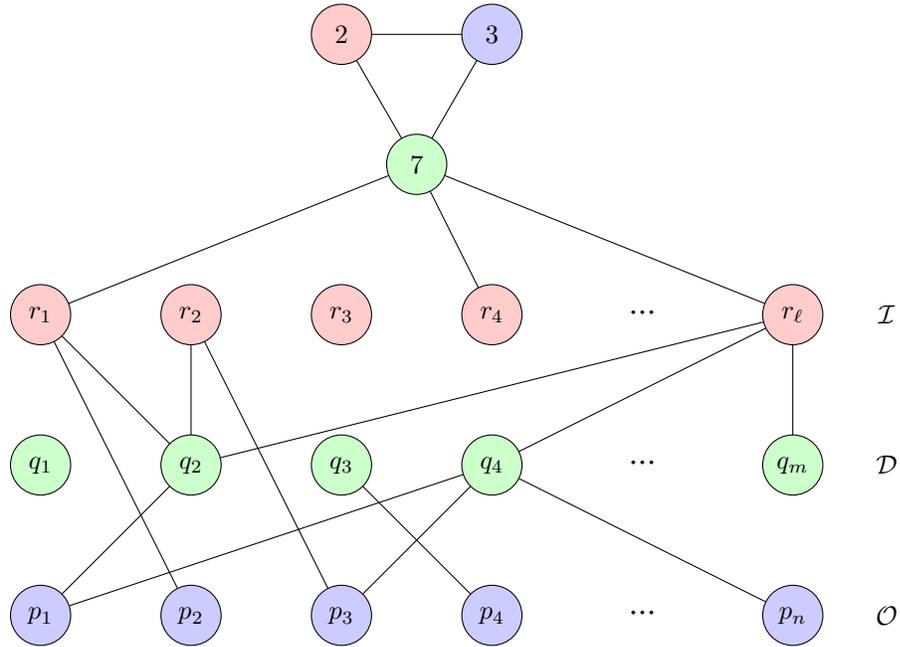

We conclude the section with several examples illustrating the classification of prime graphs of pseudo $\PSL(2,7)$-solvable groups.
\begin{example}
It is natural to ask whether the condition in Theorem \ref{thm:PSL_classification} that all the neighbors of $c$ other than $a$ and $b$ have the same color is really necessary. The graph given in Figure \ref{fig:psl_counter_example} justifies this condition. The graph may appear to be the complement of a prime graph of a pseudo $\PSL(2,7)$-solvable group, since it is $3$-colorable and its only triangle $\{a, b, c\}$ has only the vertex $c$ having neighbors outside $\{a, b, c\}$, as required by Theorem \ref{thm:PSL_classification}. However, there is no $3$-coloring of the graph such that all neighbors of $c$ outside the triangle have the same color, as the reader may readily verify. 
\end{example}

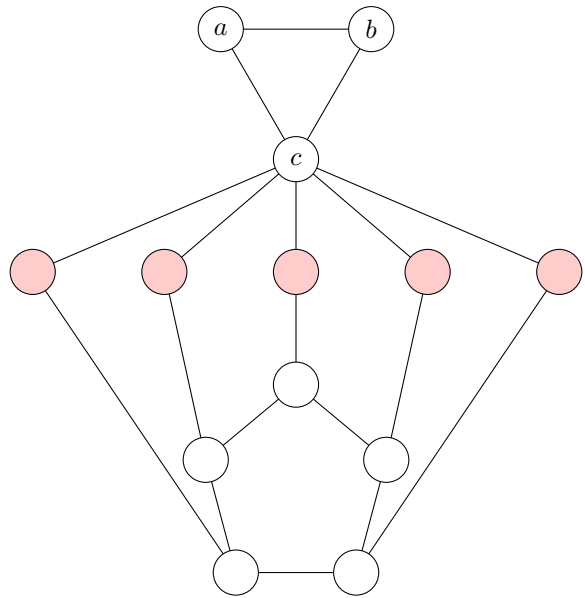
\begin{figure}
    \centering
    \begin{tikzpicture}
\node (7) at (0, 0) {$\bullet$};
\node (3) at (-1, 1.73) {$\bullet$};
\node (2) at (1, 1.73) {$\bullet$};

\node (r1) at (-3.5, -1.5) {$\bullet$};
\node (r2) at (-1.75, -1.5) {$\bullet$};
\node (r3) at (0, -1.5) {$\bullet$};
\node (r4) at (1.75, -1.5) {$\bullet$};
\node (r5) at (3.5, -1.5) {$\bullet$};

\node (p3) at (0, -3) {$\bullet$};
\node (p2) at (-1.2, -4) {$\bullet$};
\node (p4) at (1.2, -4) {$\bullet$};
\node (p1) at (-0.8, -5.5) {$\bullet$};
\node (p5) at (0.8, -5.5) {$\bullet$};

\draw (7) -- (3);
\draw (7) -- (2);
\draw (3) -- (2);

\draw (7) -- (r1);
\draw (7) -- (r2);
\draw (7) -- (r3);
\draw (7) -- (r4);
\draw (7) -- (r5);

\draw (p1) -- (r1);
\draw (p2) -- (r2);
\draw (p3) -- (r3);
\draw (p4) -- (r4);
\draw (p5) -- (r5);

\draw (p1) -- (p2);
\draw (p2) -- (p3);
\draw (p3) -- (p4);
\draw (p4) -- (p5);
\draw (p5) -- (p1);

\draw[fill=white] (2) circle (0.3);
\draw[fill=white] (3) circle (0.3);
\draw[fill=white] (7) circle (0.3);

\draw[fill=red!20!white] (r1) circle (0.3);
\draw[fill=red!20!white] (r2) circle (0.3);
\draw[fill=red!20!white] (r3) circle (0.3);
\draw[fill=red!20!white] (r4) circle (0.3);
\draw[fill=red!20!white] (r5) circle (0.3);

\draw[fill=white] (p1) circle (0.3);
\draw[fill=white] (p2) circle (0.3);
\draw[fill=white] (p3) circle (0.3);
\draw[fill=white] (p4) circle (0.3);
\draw[fill=white] (p5) circle (0.3);

\node at (3) {$a$};
\node at (2) {$b$};
\node at (7) {$c$};
\end{tikzpicture} 
    \caption{This graph is not the prime graph complement of any pseudo $\PSL(2, 7)$-solvable group.}
    \label{fig:psl_counter_example}
\end{figure}
\begin{example}
Figures \ref{fig:psl27-G1} and \ref{fig:psl27-G2} provide the prime graph complements of $G_1 = \SL(2,7)$ and $G_2 = \PSL(2,7) \times C_5$, respectively. These are both strictly pseudo $\PSL(2,7)$-solvable groups. The group $G_1$ illustrates Case (1) in the classification, while $G_2$ illustrates Case (2).
\end{example}

\begin{figure}
    \centering
    \begin{tikzpicture}
\node (7) at (0, 0) {$\bullet$};
\node (3) at (-1, 1.73) {$\bullet$};
\node (2) at (1, 1.73) {$\bullet$};

\draw (7) -- (3);

\draw[fill=white] (2) circle (0.3);
\draw[fill=white] (3) circle (0.3);
\draw[fill=white] (7) circle (0.3);

\node at (3) {$3$};
\node at (2) {$2$};
\node at (7) {$7$};
\end{tikzpicture} 
    \caption{$\graphcomp(\SL(2,7))$}
    \label{fig:psl27-G1}
\end{figure}
\begin{figure}
    \centering
    \begin{tikzpicture}
\node (7) at (0, 0) {$\bullet$};
\node (3) at (-1, 1.73) {$\bullet$};
\node (2) at (1, 1.73) {$\bullet$};
\node (5) at (2, 0) {$\bullet$};

\draw (7) -- (3);
\draw (7) -- (2);
\draw (2) -- (3);

\draw[fill=white] (2) circle (0.3);
\draw[fill=white] (3) circle (0.3);
\draw[fill=white] (7) circle (0.3);
\draw[fill=white] (5) circle (0.3);

\node at (3) {$3$};
\node at (2) {$2$};
\node at (7) {$7$};
\node at (5) {$5$};
\end{tikzpicture} 
    \caption{$\graphcomp(\PSL(2,7) \times C_5)$}
    \label{fig:psl27-G2}
\end{figure}


We now consider the unitary group $U_3(3)$. This is another nonabelian simple group such that $\pi(U_3(3)) = \{2, 3, 7\}$. Fortunately, $\PSL(2, 7)$ is a subgroup of $U_3(3)$, so we can exploit the classification of prime graphs of pseudo $\PSL(2, 7)$-solvable groups to classify the prime graphs of pseudo $U_3(3)$-solvable groups.

\begin{theorem}\label{thm:U33_classification}
A graph $\Gamma$ is isomorphic to the prime graph of a pseudo $U_3(3)$-solvable group exactly when $\graphcomp$ is $3$-colorable and triangle-free.
\end{theorem}

\begin{proof}
Suppose $\graphcomp$ is $3$-colorable and triangle free. Then we can choose a solvable group $G$ such that $\graphcomp \cong \graphcomp(G)$. This group $G$ is also pseudo $U_3(3)$-solvable, and so $\Gamma$ is isomorphic to the prime graph of a pseudo $U_3(3)$-solvable group. 

Now suppose $\Gamma = \Gamma(G)$ for some pseudo $U_3(3)$-solvable group $G$. If $G$ is solvable, then $\graphcomp$ is already $3$-colorable and triangle-free by Lemma \ref{lem:solvablegroups}. Now suppose that $G$ is strictly pseudo $U_3(3)$-solvable. We apply Corollary \ref{cor:subgroupcor} to retrieve a subgroup $K_1 = N.U_3(3) \leq G$. In particular, $N$ is solvable and $\pi(K_1) = \pi(G)$. Since $\PSL(2, 7) \leq U_3(3)$, there is a further subgroup $K_2 = N.\PSL(2, 7) \leq K_1 \leq G$. Since $N$ is solvable, $K_2$ is pseudo $\PSL(2, 7)$-solvable, and since $\pi(\PSL(2, 7)) = \pi(U_3(3))$, then $\pi(K_2) = \pi(G)$. Thus, $\graphcomp(G)$ may be reached by removing edges from $\graphcomp(K_2)$. 

Since $K_2$ is pseudo $\PSL(2, 7)$-solvable, then by Theorem \ref{thm:PSL_classification}, $\graphcomp(K_2)$ is $3$-colorable and has at most one triangle, which must be the $\{2, 3, 7\}$ triangle if it exists. Since $\graphcomp(G)$ may be realized by removing edges from $\graphcomp(K_2)$, then $\graphcomp(G)$ must be $3$-colorable as well, and must have at most one triangle, which will be the $\{2, 3, 7\}$ triangle if it exists. Note that $U_3(3)$ has an element of order $6$. Therefore, so does $G$, and thus $\edge{2}{3} \notin \graphcomp(G)$. Hence, no $\{2, 3, 7\}$ triangle exists in $\graphcomp(G)$. We conclude that $\graphcomp(G)$ is $3$-colorable and triangle-free.
\end{proof}

\section{The Alternating Group $A_6$ and the Unitary Group $U_4(2)$}\label{section:a6}

Now, we consider the alternating group $A_6$, and the prime graphs of pseudo $A_6$-solvable groups. Like $\PSL(2, 7)$, we can show that pseudo $A_6$-solvable groups have prime graph complements that are $3$-colorable with at most one triangle. We then prove that this triangle must be isolated if it exists, which leads directly to a classification result. With this result, we also achieve a classification of prime graphs of pseudo $U_4(2)$-solvable groups, which uses the fact that $A_6 \leq U_4(2)$ and that $U_4(2)$ contains an order $6$ element. However, we must first handle the case of $A_6$, which rests largely upon the following two lemmas.

\begin{lemma}\label{lem:A6_remove_3-5}
Let $G$ be a pseudo $A_6$-solvable group. Then $\graphcomp(G) \setminus \{\edge{3}{5}\}$ is $3$-colorable and triangle-free.
\end{lemma}

\begin{proof}
Let $p = 5, q = 2, r = 3$. Apply Lemma \ref{lem:remove_p-r_then_solvable} with subgroups $D_{10}, A_4 \leq A_6$. 
\end{proof}

\begin{lemma}\label{lem:A6_3-p}
Let $G$ be a strictly pseudo $A_6$-solvable group, and let $p \notin \{2,3,5\}$. Then $\edge{3}{p} \notin \graphcomp(G)$.
\end{lemma}

\begin{proof}
Suppose that $\edge{3}{p}$ is an edge in $\graphcomp(G)$. By Corollary \ref{cor:subgroupcor}, we have a subgroup $K = N.A_6$ of $G$, with $N$ solvable, such that $\pi(G) = \pi(K)$. Note that $N.C_3^2 \leq K$ is solvable, so it has a Hall $\{3,p\}$-subgroup $H_{3p}$. By Lemma \ref{lem:Hallsubgroups}, $H_{3p} = A.C_3^2$, where $A$ is the Hall $\{3,p\}$-subgroup of $N$. 

Since $\edge{3}{p} \in \graphcomp(G)$, we have $\edge{3}{p} \in \graphcomp(H_{3p})$. Thus, $H_{3p}$ is Frobenius or $2$-Frobenius. Given the structure $H_{3p} = A.C_3^2$, this group must be Frobenius of type $(3,p)$ or $2$-Frobenius of type $(3,p,3)$. In particular, there exists a $3$-group which is a Frobenius complement with quotient $C_3^2$. But the only $3$-groups which can act Frobeniusly are cyclic, and these do not have noncyclic quotients. This contradiction proves $\edge{3}{p} \notin \graphcomp(G)$.
\end{proof}

\begin{corollary}\label{cor:A6-onetriangle}
If $G$ is pseudo $A_6$-solvable, then $\graphcomp(G)$ has at most one triangle, which must be the triangle $\{2, 3, 5\}$ if it exists.
\end{corollary}

\begin{proof}
By Lemma \ref{lem:A6_remove_3-5}, the only triangles that are possible in $\graphcomp(G)$ are $\{2, 3, 5\}$ and $\{3, 5, p\}$ for some $p \notin \{2,3,5\}$. However, Lemma \ref{lem:A6_3-p} eliminates the possibility of a $\{3, 5, p\}$ triangle, so we are done.
\end{proof}

\begin{corollary}\label{cor:A6-3colorable}
If $G$ is a pseudo $A_6$-solvable group, then $\graphcomp(G)$ is $3$-colorable. 
\end{corollary}

\begin{proof}
The case where $G$ is solvable follows from Lemma \ref{lem:solvablegroups}, so assume $G$ is \textit{strictly} pseudo $A_6$-solvable. In $\graphcomp(G)$, the only possible neighbors of the vertex $3$ are $2$ and $5$ by Lemma \ref{lem:A6_3-p}. So the vertex $3$ has degree at most $2$. Thus, in order to show that $\graphcomp(G)$ is $3$-colorable, it suffices to show that $\graphcomp(G) \setminus \{3\}$ is $3$-colorable. However, $\graphcomp(G) \setminus \{3\}$ is $3$-colorable, being a subgraph of $\graphcomp(G) \setminus \{\edge{3}{5}\}$, which is $3$-colorable by Lemma \ref{lem:A6_remove_3-5}. Thus, $\graphcomp(G)$ is $3$-colorable.
\end{proof}

We have established some basic facts about the colorability and number of triangles in prime graph complements of pseudo $A_6$-solvable groups, so now we move towards a more general classification. If $G$ is pseudo $A_6$-solvable and $\{2, 3, 5\}$ does not form a triangle in $\graphcomp(G)$, we may invoke Corollaries \ref{cor:A6-onetriangle} and \ref{cor:A6-3colorable} to show that $\graphcomp(G)$ is $3$-colorable and triangle-free. Therefore, the relevant case is when $\{2, 3, 5\}$ does form a triangle in $\graphcomp(G)$. We study the restrictions a $\{2, 3, 5\}$ triangle imposes on the structure of $G$, and use these restrictions to show that the $\{2, 3, 5\}$ triangle must be isolated in $\graphcomp(G)$.

\begin{lemma}\label{lem:A6_subgrouplemma}
If $G$ is pseudo $A_6$-solvable and $\{2, 3, 5\}$ forms a triangle in $\graphcomp(G)$, then $G$ has a subgroup $K = N.A_6$ where $N$ is solvable, $\pi(K) = \pi(G)$, and $(|N|, 15) = 1$. 
\end{lemma}

\begin{proof}
By Corollary \ref{cor:subgroupcor}, $G$ has a subgroup $K = N.A_6$ where $N$ is solvable and $\pi(K) = \pi(G)$. 

If $3 \mid |N|$, then consider the subgroup $K_1 = N.D_{10}$ of $K$. This is solvable, and $\pi(K_1) = \pi(G)$. Therefore $\graphcomp(G)$ may be reached by removing edges from $\graphcomp(K_1)$. Since $K_1$ is solvable, then $\graphcomp(K_1)$ is $3$-colorable and triangle-free, and so $\graphcomp(G)$ is as well. This contradicts $\{2, 3, 5\}$ forming a triangle in $\graphcomp(G)$, and so $3 \nmid |N|$. 

Similarly if $5 \mid |N|$, consider the subgroup $K_2 = N.S_4$ of $K$. This is solvable, and $\pi(K_2) = \pi(G)$, so by the same argument as above, we reach a contradiction. So $5 \nmid |N|$.
\end{proof}



We will now apply Lemma \ref{lem:extensions_of_A6} with $A_6 = \PSL(2,9)$ and $p = 2$.

\begin{proposition}\label{prop:A6_5-p}
If $G$ is pseudo $A_6$-solvable and $\{2, 3, 5\}$ forms a triangle in $\graphcomp(G)$, then there cannot be an edge $\edge{5}{p}$ in $\graphcomp(G)$ for $p \notin \{2, 3, 5\}$. 
\end{proposition}

\begin{proof}
By Lemma \ref{lem:A6_subgrouplemma}, $G$ has a subgroup $K = N.A_6$ where $N$ is solvable, $\pi(K) = \pi(G)$, and $(|N|, 15) = 1$. 

Set $\pi = \{2, 3, 5, p\}$ and note that $K$ is $\pi$-separable. Thus, we can take a Hall $\pi$-subgroup $H_\pi$ of $K$. By Lemma \ref{lem:Hallsubgroups}, $H_\pi = B.A_6$, where $B$ is a Hall $\{2, p\}$-subgroup of $N$. Consider the chief series of $B$, which takes the form $B = B_1.B_2. \cdots . B_{k-1}.B_k$. As a property of the chief series, each chief factor $B_i$ is either a $p$-group or a $2$-group. 

If $p \nmid |B|$, then $p \notin \pi(K) = \pi(G)$, so clearly $\edge{5}{p} \notin \graphcomp(G)$, as desired. Otherwise, if $p \mid |B|$, then $B$ is nontrivial, so $k \geq 1$. We claim $B_k$ is a $p$-group. If instead $B_k$ were a nontrivial $2$-group, then $B_k.A_6$ contains an element of order $6$ by Lemma \ref{lem:extensions_of_A6}. Since $B_k.A_6$ is a quotient of the subgroup $H_\pi$ of $G$, then $G$ contains an element of order $6$ as well. This contradicts the triangle $\{2, 3, 5\}$ in $\graphcomp(G)$. Therefore, $B_k$ is indeed a $p$-group. 

Consider $B_k.A_6$. As reported in Table \ref{table:2}, we determined computationally that every complex irreducible representation $\rho$ of $A_6$ has the property that for every $t \in A_6$ with $|t| = 5$, $\rho(t)$ has a fixed point. Since the extension $B_k.A_6$ is coprime, we may apply Lemma \ref{lem:rep_theory_biz} to see that $B_k.A_6$ has an element of order $5p$. Therefore, since $B_k.A_6$ is a quotient of a subgroup of $G$, $G$ has an element of order $5p$. Therefore, $\edge{5}{p} \notin \graphcomp(G)$, as desired.
\end{proof}

We are now ready for our main classification result for $A_6$.

\begin{theorem}\label{thm:A6-classification}
Let $\Gamma$ be a graph. Then the following are equivalent:
\begin{enumerate}
    \item $\Gamma$ is the prime graph of a pseudo $A_6$-solvable group $G$.
    \item $\graphcomp$ is $3$-colorable and has at most one triangle, which must be isolated if it exists.
\end{enumerate}
\end{theorem}

\begin{proof}
Suppose condition (1) holds for $\Gamma$, i.e., $\graphcomp = \graphcomp(G)$ for some pseudo $A_6$-solvable group $G$. Then by Corollary \ref{cor:A6-3colorable}, $\graphcomp(G)$ is $3$-colorable, and by Corollary \ref{cor:A6-onetriangle}, $\graphcomp(G)$ has at most one triangle, which must be $\{2, 3, 5\}$ if it exists. If $\graphcomp(G)$ does not contain a triangle, then we are done. If $\graphcomp(G)$ does contain a triangle, it remains to show that the vertices $2$, $3$, and $5$ are not connected to anything outside $\{2, 3, 5\}$. 

Since $\{2, 3, 5\}$ is a triangle in $\graphcomp(G)$, then $G$ must be strictly pseudo $A_6$-solvable. By Lemma \ref{lem:A6_3-p}, the vertex $3$ is not connected to anything outside $\{2, 3, 5\}$. Likewise, by Proposition \ref{prop:A6_5-p}, the vertex $5$ is not connected to anything outside $\{2, 3, 5\}$. By Corollary \ref{cor:subgroupcor}, $G$ contains a subgroup $K = N.A_6$ where $N$ is solvable, and $\pi(K) = \pi(G)$. This group contains a further subgroup $L = N.A_5$, which is pseudo $A_5$-solvable. It also holds that $\pi(L) = \pi(G)$, and so $\graphcomp(G)$ can be realized by removing edges from $\graphcomp(L)$. Since $\graphcomp(G)$ has a $\{2, 3, 5\}$ triangle, then $\graphcomp(L)$ must have one as well. By \cite[Theorem 6.8]{A5_paper}, there is no edge $\edge{2}{p}$ in $\graphcomp(L)$ for $p \notin \{2, 3, 5\}$. Since we may arrive at $\graphcomp(G)$ by removing edges from $\graphcomp(L)$, then $\graphcomp(G)$ has no $\edge{2}{p}$ edges either. Thus, the $\{2, 3, 5\}$ triangle is isolated in $\graphcomp(G)$, and (2) holds. 

Now assume that (2) is true for $\Gamma$. If $\graphcomp$ contains no triangle, then by \cite[Theorem 2.8]{solvable_groups}, we may construct a solvable group $G$ such that $\graphcomp(G) = \graphcomp$. By construction, $G$ is also pseudo $A_6$-solvable, and so we are done in this case. 

Now suppose that $\graphcomp$ does have a triangle. Label the triangle $\{a, b, c\}$. Using the techniques in \cite[Theorem 2.8]{solvable_groups}, we can construct a solvable group $K$ such that $(|K|, |A_6|) = 1$ and $\graphcomp \setminus \{a, b, c\} \cong \graphcomp(K)$. Since the triangle $\{a, b, c\}$ is isolated, and $\graphcomp(A_6)$ is a triangle, we have $\graphcomp = \graphcomp(K \times A_6)$. Thus, condition (1) holds.
\end{proof}


We now turn our attention to the unitary group $U_4(2)$, a $K_3$-group which contains $A_6$. This containment allows us to impose restrictions on prime graphs of pseudo $U_4(2)$-solvable groups, and ultimately achieve a classification result. 

\begin{theorem}\label{thm:U42_classification}
A graph $\Gamma$ is isomorphic to the prime graph of a pseudo $U_4(2)$-solvable group exactly when $\graphcomp$ is $3$-colorable and triangle-free.
\end{theorem}

\begin{proof}
Suppose $\graphcomp$ is $3$-colorable and triangle-free. By \cite[Theorem 2.8]{solvable_groups}, we can choose a solvable group $G$ such that $\graphcomp \cong \graphcomp(G)$. This group $G$ is trivially pseudo $U_4(2)$-solvable, so $\Gamma$ is isomorphic to the prime graph of a pseudo $U_4(2)$-solvable group. 

Now suppose $\Gamma = \Gamma(G)$ for some pseudo $U_4(2)$-solvable group $G$. If $G$ is solvable, then $\graphcomp$ is already $3$-colorable and triangle-free by Lemma \ref{lem:solvablegroups}, so assume $G$ is strictly pseudo $U_4(2)$-solvable. By Corollary \ref{cor:subgroupcor}, $G$ has a subgroup $K = N.U_4(2) \leq G$, where $\pi(K) = \pi(G)$ and $N$ is solvable. Taking the pseudo $A_6$-solvable subgroup $K_1 = N.A_6$ of $K$, we observe that $\pi(K_1) = \pi(K)$. By Corollary \ref{cor:A6-3colorable}, $\graphcomp(K_1)$ is $3$-colorable. By Corollary \ref{cor:A6-onetriangle}, $\graphcomp(K_1)$ has at most one triangle, which must be the $\{2,3,5\}$ triangle if it exists. Since $\pi(K_1) = \pi(G)$ and $K_1 \leq G$, $\graphcomp(G)$ is obtained by removing edges from $\graphcomp(K_1)$. Thus, $\graphcomp(G)$ is $3$-colorable and contains at most one triangle, which must be $\{2,3,5\}$. But $G$ is strictly pseudo $U_4(2)$-solvable and $U_4(2)$ contains an element of order $6$, so $\edge{2}{3} \notin \graphcomp(G)$. This eliminates the possibility of a $\{2,3,5\}$ triangle, so $\graphcomp(G)$ is triangle-free.
\end{proof}

\section{The Projective Special Linear Group $\PSL(2,8)$}\label{section:psl28}

We now turn to the projective special linear group $\PSL(2, 8)$, which we handle similarly to the case of $A_6$. We first consider $\edge{2}{p}$, $\edge{3}{p}$, and $\edge{7}{p}$ edges, for $p \notin \{2, 3, 7\}$, showing that if $\{2, 3, 7\}$ forms a triangle in $\graphcomp(G)$, then no such edges can exist. In handling the case of $\edge{7}{p}$ edges, we also show that prime graph complements of pseudo $\PSL(2, 8)$-solvable groups are $3$-colorable and have at most one triangle, which must be $\{2, 3, 7\}$ if it exists. This flows cleanly into a classification of prime graphs of pseudo $\PSL(2, 8)$-solvable groups.

We consider the $\edge{2}{p}$ and $\edge{3}{p}$ edges first, since the arguments involved are shorter.

\begin{lemma}\label{lem:psl28_no_3-7}
Let $G$ be a pseudo $\PSL(2,8)$-solvable group. Then $\graphcomp(G) \setminus \{\edge{3}{7}\}$ is $3$-colorable and triangle-free.
\end{lemma}

\begin{proof}
Let $p=3$, $q=2$, $r=7$. Apply Lemma \ref{lem:remove_p-r_then_solvable} with subgroups $D_{18}, D_{14} \leq G$.
\end{proof}

\begin{lemma}\label{lem:psl28_subgrouplemma}
Let $G$ be strictly pseudo $\PSL(2,8)$-solvable. Then $G$ contains a subgroup $K \cong N.\PSL(2,8)$ such that $\graphcomp(G)$ is $3$-colorable and triangle-free if $3$ or $7$ divides $|N|$. 
\end{lemma}

\begin{proof}
It is immediate from Corollary \ref{cor:subgroupcor} that $G$ has a subgroup $K \cong N.\PSL(2,8)$ with $\pi(G)=\pi(K)$. If $3 \mid |N|$, then $K$ has a solvable subgroup isomorphic to $N.D_{14}$ whose order has the same prime divisors as $G$. This implies that $\graphcomp(G)$ is $3$-colorable and triangle-free. Likewise if $7 \mid |N|$ then $N.D_{18} \leq K$ with the same prime divisors so $\graphcomp(G)$ is $3$-colorable and triangle-free.
\end{proof}

\begin{lemma}\label{lem:psl28_no_3p}
Let $G$ be pseudo $\PSL(2, 8)$-solvable and $p \notin \{2, 3, 7\}$. If a $\{2, 3, 7\}$ triangle exists in $\graphcomp(G)$, then $\edge{3}{p}$ is not an edge in $\graphcomp(G)$. 
\end{lemma}

\begin{proof}
Since $\graphcomp(G)$ contains a triangle, $G$ is not solvable. Then by Lemma \ref{lem:psl28_subgrouplemma}, $G$ has a subgroup $K = N.\PSL(2, 8)$ such that $N$ is solvable and $\pi(G) = \pi(K)$. Note that since we add edges to $\graphcomp(G)$ to get to $\graphcomp(K)$, then $\{2, 3, 7\}$ forms a triangle in $\graphcomp(K)$. Lemma \ref{lem:psl28_subgrouplemma} further states that if $3 \mid |N|$ or $7 \mid |N|$, then $\graphcomp(G)$ is triangle-free, a contradiction. So $3, 7 \nmid |N|$. 

Let $\pi = \{2,3,7,p\}$. Since $K$ is $\pi$-separable, we can take a Hall $\pi$-subgroup $H_{\pi}$. By Lemma \ref{lem:Hallsubgroups}, we can write $H_{\pi} = A.\PSL(2,8)$, where $A$ is a Hall $\{2, p\}$-subgroup of $N$. Since $p \mid |A|$, $A$ is not a $2$-group. Then by \cite[Proposition 4.2]{3-centralizers}, every element of order $3$ in $K$ has a fixed point acting on $A$. Choose $t \in \PSL(2,8)$ of order $3$. Since $A$ is a $\{2, p\}$-group, then $t$ must fix an element of order $2$ or an element of order $p$. Since $\edge{2}{3} \in \graphcomp(K)$, $t$ must fix an element of order $p$, eliminating the edge $\edge{3}{p}$ from $\graphcomp(K)$. Therefore, $\edge{3}{p} \notin \graphcomp(G)$.
\end{proof}

\begin{lemma}\label{lem:psl28_no_2-p}
Let $G$ be a strictly pseudo $\PSL(2,8)$-solvable group and $p \notin \{2,3,7\}$. Then $\edge{2}{p} \notin \graphcomp(G)$. 
\end{lemma}

\begin{proof}
By Corollary \ref{cor:subgroupcor}, $G$ has a subgroup $K \cong N.\PSL(2,8)$ where $N$ is solvable and $\pi(K)=\pi(G)$. Suppose by way of contradiction that $\edge{2}{p} \in \graphcomp(G)$. Now, $K$ contains the subgroup $N.C_2^3$ which is solvable and thus has a Hall $\{2,p\}$-subgroup $H_{2p}$. By Lemma \ref{lem:Hallsubgroups}, we can write $H_{2p} = A.C_2^3$, where $A$ is a Hall $\{2,p\}$-subgroup of $N$. Since $\edge{2}{p} \in \graphcomp(G)$, then $\edge{2}{p} \in \graphcomp(H_{2p})$, so $H_{2p}$ is Frobenius or $2$-Frobenius. Moreover, since $H_{2p} = A.C_2^3$, $H_{2p}$ is Frobenius of type $(2,p)$ or $2$-Frobenius of type $(2,p,2)$. Either way, there exists a $2$-group which is a Frobenius complement with quotient $C_2^3$. 

The only $2$-groups which can act Frobeniusly are cyclic or generalized quaternion. But no cyclic group has a noncyclic quotient, and the largest abelian quotient group of a generalized quaternion group is $V_4$. This contradiction implies $\edge{2}{p} \notin \graphcomp(G)$. 
\end{proof}

We have now shown that if $G$ is pseudo $\PSL(2, 8)$-solvable and if $\graphcomp(G)$ contains a $\{2, 3, 7\}$-triangle, then $2$ and $3$ cannot have any neighbors outside the triangle. The same remains to be shown for $7$. Although the argument is longer, it is also more fruitful, providing as corollaries a description of the type and number of triangles in $\graphcomp(G)$, and the fact that $\graphcomp(G)$ is always $3$-colorable. 

\begin{proposition}\label{prop:psl28_no_7-p_or_3-col}
Let $G$ be a pseudo $\PSL(2, 8)$-solvable group. Then at least one of the following holds:
\begin{enumerate}
    \item $\graphcomp(G)$ is $3$-colorable and triangle-free
    \item $\edge{7}{p}$ is not an edge in $\graphcomp(G)$ for any $p \notin \{2, 3, 7\}$.
\end{enumerate}
\end{proposition}

\begin{proof}
If $G$ is solvable, then by Lemma \ref{lem:solvablegroups} we are done. So assume $G$ is strictly pseudo $\PSL(2,8)$-solvable. By Lemma \ref{lem:psl28_subgrouplemma}, $G$ has a subgroup $K = N.\PSL(2, 8)$ such that $N$ is solvable and $\pi(G) = \pi(K)$. Lemma \ref{lem:psl28_subgrouplemma} further states that if $3 \mid |N|$ or $7 \mid |N|$, then $\graphcomp(G)$ is $3$-colorable and triangle-free, and so we are done in either of those cases. Therefore, we may assume $3, 7 \nmid |N|$.

Let $p \notin \{2, 3, 7\}$. Setting $\pi = \{2, 3, 7, p\}$, we may take a Hall $\pi$-subgroup $H_\pi \leq K$, which by Lemma \ref{lem:Hallsubgroups} takes the form $H_\pi = H_{2p}.\PSL(2, 8)$, where $H_{2p}$ is a Hall $\{2, p\}$-subgroup of $N$. By considering the chief series of $H_\pi$, we get that $H_\pi$ has a quotient $A = V.B.\PSL(2, 8)$, where $V$ is a nontrivial elementary abelian $p$-group and $B$ is a possibly trivial $2$-group. 

If we set $L = B.\PSL(2, 8)$, then $V$ is a faithful, irreducible $L$-module over $\F_p$. Recall that $\PSL(2, 8)$ has a Hall $\{2, 7\}$-subgroup $S \cong C_2^3 \rtimes C_7$, where $C_7$ acts on $C_2^3$ Frobeniusly. Then $L$ has a subgroup $H = B.S = B.C_2^3.C_7$. Then, $V$ is a faithful $H$-module over $\F_p$. Let $Q \in \on{Syl}_2(H)$ and let $R \in \on{Syl}_7(H)$. Then $Q = B.C_2^3$ and $H = Q.R = QR$, the last equality coming from Schur-Zassenhaus.

Suppose for contradiction that $R$ acts Frobeniusly on $V$. Then $C_V(R)$ is trivial. By Maschke's Theorem, since $H$ acts coprimely on $V$, then $V$ is completely reducible, so $V = V_1 \oplus V_2 \oplus \cdots \oplus V_\ell$, where each $V_i$ is an irreducible $H$-module, and $\ell$ is a positive integer.

Let $K_i \unlhd H$ be the kernel of the action of $H$ on $V_i$, for $i = 1, \dots, \ell$. Then $H/K_i$ acts faithfully and irreducibly on $V_i$. Let $K$ be the kernel of the action of $H$ on $V$. Since $V$ is a faithful $H$-module, then $\bigcap_{i=1}^\ell K_i \leq K = 1$. Now,
\[
    H \cong H / \bigcap_{i=1}^\ell K_i \lesssim H/K_1 \times H/K_2 \times \cdots \times H/K_\ell.
\]
Note that since $R$ acts Frobeniusly on $V$, then $R \cap K_i = 1$ for each $i = 1, \dots, \ell$. Assume for contradiction that $R \cong RK_i / K_i$ acts trivially on $Q_i \in \on{Syl}_2(H/K_i)$ for every $i = 1, \dots, \ell$. Then $H / K_i \cong R \times Q_i$, and so, 
\[
    H \lesssim \bigtimes_{i=1}^\ell H/K_i \cong \bigtimes_{i=1}^\ell (R \times Q_i) \cong R^\ell \times \left(\bigtimes_{i=1}^\ell Q_i\right). 
\]
Since $H = RQ \lesssim R^\ell \times \bigtimes_{i=1}^\ell Q_i$, 
every element of $R$ commutes with every element of $Q$, so $RQ = R \times Q$. Thus, $R$ acts trivially on $Q$, which is a contradiction, since $R$ acts nontrivially on $C_2^3 \leq \PSL(2, 8)$, and $Q.R = B.C_2^3.R$. Thus, $R$ must act faithfully on some $Q_i$, say $Q_1$. 

Since $R$ acts nontrivially on $Q_1$, then $[R, Q_1] \neq 1$. However, $RQ_1 = H / K_1$ acts faithfully and irreducibly on $V_1$, and $C_{V_1}(R) \leq C_V(R) = 0$. Then by \cite[Theorem A]{Flavell}, $[Q_1, R] = 1$, since $7$ is not a Fermat prime. This is a contradiction. Thus, $R$ does not act Frobeniusly on $V$. This gives us an element of order $7p$ in $V.H$ which implies there exists an element of order $7p$ in $G$. Thus, $\edge{7}{p} \notin \graphcomp(G)$.
\end{proof}

\begin{corollary}\label{cor:psl28_only_2_3_7}
Let $G$ be pseudo $\PSL(2, 8)$-solvable. Then $\{2, 3, 7\}$ is the only possible triangle in $\graphcomp(G)$. 
\end{corollary}
\begin{proof}
By Lemma \ref{lem:psl28_no_3-7}, the only triangles possible in $\graphcomp(G)$ are $\{2, 3, 7\}$ and $\{3, 7, p\}$ for $p \notin \{2, 3, 7\}$. However, a $\{3, 7, p\}$ triangle contradicts Proposition \ref{prop:psl28_no_7-p_or_3-col}, so the only possible triangle in $\graphcomp(G)$ is $\{2, 3, 7\}$.
\end{proof}

\begin{corollary}\label{cor:psl28_no_7-p}
Let $G$ be pseudo $\PSL(2, 8)$-solvable. If a $\{2, 3, 7\}$ triangle exists in $\graphcomp(G)$, then $\edge{7}{p}$ is not an edge in $\graphcomp(G)$.
\end{corollary}

\begin{corollary}\label{cor:psl28_3-col}
Let $G$ be pseudo $\PSL(2, 8)$-solvable. Then $\graphcomp(G)$ is $3$-colorable.
\end{corollary}

\begin{proof}
By Proposition \ref{prop:psl28_no_7-p_or_3-col}, $\graphcomp(G)$ is either $3$-colorable and triangle-free or it has no $\edge{7}{p}$ edge for each $p \notin \{2, 3, 7\}$. If $\graphcomp(G)$ is $3$-colorable and triangle-free, we are done. Otherwise, the only vertices $7$ can be connected to are $3$ and $2$. Thus, the vertex $7$ has degree at most $2$. So in order to show that $\graphcomp(G)$ is $3$-colorable, it suffices to show that $\graphcomp(G) \setminus \{7\}$ is $3$-colorable. However, $\graphcomp(G) \setminus \{7\}$ is a subgraph of $\graphcomp(G) \setminus \{\edge{3}{7}\}$, which is $3$-colorable by Lemma \ref{lem:psl28_no_3-7}. A subgraph of a $3$-colorable graph is $3$-colorable, so we are done.
\end{proof}

If $G$ is pseudo $\PSL(2, 8)$-solvable, we now know that $\graphcomp(G)$ must contain at most one triangle, and applying Lemmas \ref{lem:psl28_no_3p} and \ref{lem:psl28_no_2-p}, as well as Corollary \ref{cor:psl28_no_7-p}, we know that such a triangle must be isolated. Along with the fact that $\graphcomp(G)$ is always $3$-colorable, we now obtain a classification.

\begin{theorem}
Let $\Gamma$ be a graph. Then the following are equivalent:
\begin{enumerate}
    \item $\Gamma$ is the prime graph of a pseudo $\PSL(2, 8)$-solvable group $G$.
    \item $\graphcomp$ is $3$-colorable and has at most one triangle, which must be isolated if it exists.
\end{enumerate}
\end{theorem}

\begin{proof}
Suppose (1) holds. So $\Gamma \cong \Gamma(G)$ for some pseudo $\PSL(2, 8)$-solvable group $G$. By Corollary \ref{cor:psl28_3-col}, $\graphcomp(G)$ is already $3$-colorable, so it suffices to show that $\graphcomp(G)$ has at most one triangle, which must be isolated if it exists.

If $\graphcomp(G)$ is triangle-free, we are done, so we may assume it has a triangle. By Corollary \ref{cor:psl28_only_2_3_7}, $\graphcomp(G)$ must have at most one triangle which must be $\{2, 3, 7\}$. For any $p \notin \{2, 3, 7\}$, Lemmas \ref{lem:psl28_no_3p} and \ref{lem:psl28_no_2-p} and Corollary \ref{cor:psl28_no_7-p} show that there are no $\edge{2}{p}$, $\edge{3}{p}$, or $\edge{7}{p}$ edges. Since this holds for any $p \notin \{2, 3, 7\}$, the $\{2, 3, 7\}$ triangle must be isolated, and (2) holds. 

Now, assume (2) holds to begin with. If $\graphcomp$ contains no triangle, then by \cite[Theorem 2.8]{solvable_groups}, we may construct a solvable group $G$ such that $\graphcomp(G) = \graphcomp$. By construction, $G$ is also pseudo $\PSL(2, 8)$-solvable, and so we are done in this case. 

Now suppose that $\graphcomp$ does have a triangle. Label the triangle $\{a, b, c\}$. Using the techniques in \cite[Theorem 2.8]{solvable_groups}, we can construct a solvable group $K$ such that $(|K|, |\PSL(2, 8)|) = 1$ and $\graphcomp \setminus \{a, b, c\} \cong \graphcomp(K)$. Since the triangle $\{a, b, c\}$ is isolated, and $\graphcomp(\PSL(2, 8))$ is a triangle, we have $\graphcomp = \graphcomp(K \times \PSL(2, 8))$. Thus, condition (1) holds.
\end{proof}

\section{The Projective Special Linear Group $\PSL(3,3)$}\label{section:psl33}

Next, we turn to the projective special linear group $\PSL(3,3)$. Referring to Table \ref{table:1}, we see that its Sylow $2$-subgroup is the semidihedral group $SD_{16}$ which cannot act Frobeniusly. Also the prime graph of $\PSL(3,3)$ does not contain a triangle. These facts make the classification for pseudo $\PSL(3,3)$-solvable groups significantly easier than in previous sections.

After we classify the prime graphs, we conclude the section with a remark about the distinction between classifying the prime graphs of pseudo $T$-solvable groups and classifying the prime graphs of strictly pseudo $T$-solvable groups.

\begin{lemma}\label{lem:PSL(3,3) no 2-p}
Let $G$ be strictly pseudo $\PSL(3,3)$-solvable. If the edge $\edge{2}{p}$ is in $\graphcomp(G)$ for some $p \notin \{2,3,13\}$, then $\graphcomp(G)$ is $3$-colorable and triangle-free. 
\end{lemma}

\begin{proof}
By Corollary \ref{cor:subgroupcor}, $G$ has a subgroup $K \cong N.\PSL(3,3)$, where $N$ is solvable and $\pi(G) = \pi(K)$. Suppose $2 \nmid |N|$. Then by Lemma \ref{lem:Hallsubgroups}, $P.SD_{16}$ is a Hall $\{2,p\}$-subgroup of $N.SD_{16} \leq K$, where $P$ is a Sylow $p$-subgroup of $N$, and $SD_{16}$ denotes the semidihedral group of order $16$ (the Sylow $2$-subgroup of $\PSL(3,3)$). Since $SD_{16}$ cannot act Frobeniusly, there is an order $2p$ element in $P.SD_{16} \leq K \leq G$, contradicting $\edge{2}{p} \in \graphcomp(G)$. 
Thus $2$ must divide $|N|$. Therefore, $K_1 \cong N.(C_{13} \rtimes C_3)$ is a solvable subgroup of $K$ with the same set of prime divisors. This implies that $\graphcomp(K_1)$, and therefore $\graphcomp(G)$, is triangle-free and $3$-colorable.
\end{proof}

\begin{theorem}\label{thm:psl33-classification}
A simple graph $\Gamma$ is isomorphic to the prime graph of a pseudo $\PSL(3,3)$-solvable group if and only if $\graphcomp$ is $3$-colorable and triangle-free.
\end{theorem}

\begin{proof}
If $\graphcomp$ is $3$-colorable and triangle-free, then by \cite[Theorem 2.8]{solvable_groups} a suitable solvable group can be constructed which is trivially pseudo $\PSL(3,3)$-solvable. Likewise, if $G$ is solvable, $\graphcomp(G)$ will be $3$-colorable and triangle-free.

If $G$ is a strictly pseudo $\PSL(3,3)$-solvable group, then by Corollary \ref{cor:subgroupcor}, $G$ has a subgroup $K \cong N.\PSL(3,3)$ with $\pi(K)=\pi(G)$ and $N$ solvable. Note that $K$ has a solvable subgroup $K_1 \cong N.(C_{13} \rtimes C_3)$, and $\pi(K_1) \cup \{2\}=\pi(G)$. Thus $\graphcomp(G) \setminus \{2\}$ is obtained by strictly removing edges from $\graphcomp(K_1)$ and so it is $3$-colorable and triangle-free. Now take $\graphcomp(G)$ in its entirety. If the vertex $2$ is connected to a prime $p \neq 3,13$, then $\graphcomp(G)$ is 3-colorable and triangle-free by Lemma \ref{lem:PSL(3,3) no 2-p}. Also there cannot be an edge between $2$ and $3$ because $\PSL(3,3)$ and thus $G$ contains an element of order $6$. So we may assume $2$ is connected to no primes other than $13$ in $\graphcomp(G)$. Then $2$ has degree at most one, so it cannot be in a triangle, and it can not increase the chromatic number. Therefore the entirety of $\graphcomp(G)$ is $3$-colorable and triangle-free.
\end{proof}

\begin{remark}\label{rmk:Tim's-remark}
In this paper, we classify the prime graphs of pseudo $T$-solvable groups, where $T$ is a $K_3$-group. It is important to note that we do not classify the prime graphs of \textit{strictly} pseudo $T$-solvable groups. This is a necessary distinction, and with the following lemma, we will be able to give an example of an infinite family of $3$-colorable triangle-free graphs which are not isomorphic to the complement of the prime graph of any strictly pseudo $\PSL(3,3)$-solvable group. This more specific classification of strictly pseudo $T$-solvable groups remains an open problem for each $K_3$-group $T$.
\end{remark}

\begin{lemma}\label{lem:Tim'sLittleLemma}
Let $G$ be a strictly $\PSL(3,3)$-solvable group. There are no edges $\edge{3}{p}$ with $p \notin \{2,3,13\}$ in $\graphcomp(G)$.
\end{lemma}

\begin{proof}
By Corollary \ref{cor:subgroupcor}, $G$ has a subgroup $K \cong N.\PSL(3,3)$ where $N$ is solvable and $\pi(G) = \pi(K)$.
Then $K$ contains a subgroup $N.(C_3^2 \rtimes C_3)$. This subgroup is solvable so by Lemma \ref{lem:Hallsubgroups} it contains a Hall $\{3,p\}$-subgroup $H \cong A.(C_3^2 \rtimes C_3)$, where $A$ is a Hall $\{3,p\}$-subgroup of $N$. If $H$ is Frobenius or $2$-Frobenius, then this implies the existence of a Frobenius complement which is a $3$-group with $C_3^2 \rtimes C_3$ as a quotient. Since all Frobenius complements that are $3$-groups must be cyclic, this is a contradiction. Therefore, $H$ is not Frobenius or $2$-Frobenius, and so it contains an element of order $3p$. Thus, $\edge{3}{p} \notin \graphcomp(H)$, so $\edge{3}{p} \notin \graphcomp(G)$.
\end{proof}

Lemma \ref{lem:Tim'sLittleLemma} implies in particular that the prime graph of any strictly pseudo $\PSL(3, 3)$-solvable group $G$ must have a vertex with degree at most $2$. If, then, one were to construct a $3$-colorable and triangle-free graph where every vertex has degree greater than $2$, such a graph could not be the prime graph complement of a strictly pseudo $\PSL(3, 3)$-solvable group. We provide such a construction in the following example and show that it has the desired properties in the following proposition.

\begin{figure}
    \centering
    \begin{tikzpicture}
\node (c1) at (0, 0) {$\bullet$};
\node (c2) at (1.5, 0) {$\bullet$};
\node (c3) at (3, 0) {$\bullet$};
\node (c4) at (4.5, 0) {$\bullet$};

\node (d1) at (0, 1.5) {$\bullet$};
\node (d2) at (1.5, 1.5) {$\bullet$};
\node (d3) at (3, 1.5) {$\bullet$};
\node (d4) at (4.5, 1.5) {$\bullet$};

\node (e1) at (0.75, 4) {$\bullet$};
\node (e2) at (3.75, 4) {$\bullet$};

\draw (c1) -- (c4);
\draw (0, 0) arc(-180:0:2.25);

\draw (c1) -- (d1);
\draw (c2) -- (d2);
\draw (c3) -- (d3);
\draw (c4) -- (d4);

\draw (d1) -- (e1);
\draw (d2) -- (e1);
\draw (d3) -- (e1);
\draw (d4) -- (e1);

\draw (d1) -- (e2);
\draw (d2) -- (e2);
\draw (d3) -- (e2);
\draw (d4) -- (e2);

\draw[fill=red!20!white] (c1) circle (0.3);
\draw[fill=blue!20!white] (c2) circle (0.3);
\draw[fill=red!20!white] (c3) circle (0.3);
\draw[fill=blue!20!white] (c4) circle (0.3);

\draw[fill=blue!20!white] (d1) circle (0.3);
\draw[fill=red!20!white] (d2) circle (0.3);
\draw[fill=blue!20!white] (d3) circle (0.3);
\draw[fill=red!20!white] (d4) circle (0.3);

\draw[fill=green!20!white] (e1) circle (0.3);
\draw[fill=green!20!white] (e2) circle (0.3);
\end{tikzpicture}
    \caption{This graph is $3$-colorable and triangle free. However, every vertex has degree $3$ or greater, and so this graph cannot be the prime graph complement of a strictly pseudo $\PSL(3, 3)$-solvable group.}
    \label{fig:Tim's Graph}
\end{figure}
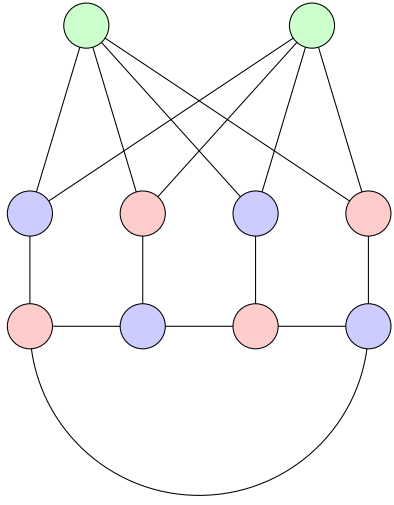

\begin{example}\label{ex:TimsExample}
Construct a graph $\Gamma$ by starting with a cycle of length $n \geq 4$. Call this cycle $C_n$. Add whiskers to the cycle. (That is, add  $n$ isolated vertices. Call this collection of vertices $V$. Connect each vertex in $V$ to a single vertex in $C_n$ such that no two vertices in $V$ connect to the same vertex in $C_n$.) Add $2$ more vertices. Call this set of vertices $K$. Connect both vertices in $K$ to all of the whiskers (i.e. draw an edge between each vertex in $V$ and each vertex in $K$). Figure \ref{fig:Tim's Graph} shows the graph resulting from this construction for $n=4$.
\end{example}

\begin{proposition}
For any value of $n$, the graph $\Gamma$, constructed in Example \ref{ex:TimsExample}, is $3$-colorable and triangle-free, but it is not isomorphic to the complement of the prime graph of any strictly pseudo $\PSL(3,3)$-solvable group $G$.
\end{proposition}

\begin{proof}
First we give a valid $3$-coloring for $\Gamma$. Start by coloring $C_n$ with red, green, and blue. All cycles can be $3$-colored. Next color the whiskers, $V$ with blue or red. This is possible because none of the whiskers are connected to each other, and each whisker connects to only one vertex on the cycle. Finally, color the two vertices in $K$ green. This is possible because they are only connected to whiskers, and none of the whiskers are green. This defines a valid $3$-coloring for any value of $n$. This coloring is demonstrated for $n=4$ in Figure \ref{fig:Tim's Graph}.

Note that for any $n \geq 4$, $C_n$ contains no triangles. Likewise the vertices in $V$ are only connected to a single vertex in $C_n$, so the induced subgraph of $\Gamma$ onto $C_n \cup V$ is still triangle free. The vertices in $K$ are disjoint, and they only connect to vertices in $V$. The vertices in $V$ are also disjoint, so the vertices in $K$ cannot be part of a triangle. Thus the entire graph is triangle-free for all $n \geq 4$.

For all values of $n$, each vertex in $C_n$ is also connected to a whisker so it has degree 3. Each vertex in $V$ has degree 3 because it is connected to one vertex in $C_n$ and the two vertices in $K$. Both vertices in $K$ are connected to the $n$ vertices in $V$. Thus for any $n \geq 4$, The vertices in $K$ have degree greater than $3$. But by Lemma \ref{lem:Tim'sLittleLemma}, in the complement of the prime graph of a pseudo $\PSL(3,3)$-solvable group $G$, the vertex $3$ can have degree at most $2$. Therefore $\Gamma$ is not the complement of the prime graph of any pseudo $\PSL(3,3)$-solvable group.
\end{proof}

\section{The Projective Special Linear Group $\PSL(2,17)$}\label{section:psl217}

The classification of prime graphs of pseudo $\PSL(2, 17)$-solvable groups is considerably more difficult than any of the preceding groups. Although we achieve the same classification as $\PSL(2, 7)$, we use different techniques and higher technology in order to arrive at this result. We begin with a familiar lemma:

\begin{lemma}\label{lem:psl217_no_3-17}
Let $G$ be a pseudo $\PSL(2,17)$-solvable group. Then $\graphcomp(G) \setminus \{\edge{3}{17}\}$ is $3$-colorable and triangle-free.
\end{lemma}
\begin{proof}
Let $p = 3, q = 2, r = 17$. Apply Lemma \ref{lem:remove_p-r_then_solvable} with the subgroups $D_{18}, D_{34} \leq \PSL(2, 17)$.
\end{proof}

From this proof we can define the Frobenius digraph of a pseudo $\PSL(2,17)$-solvable group. 

\begin{definition}
If $G$ is a pseudo $\PSL(2,17)$-solvable group, we say its \textit{Frobenius digraph} $\digraph(G)$ is the orientation of $\graphcomp(G)$ given by the orientation of $\digraph(K)$ as defined in the proof of Lemma \ref{lem:psl217_no_3-17} (which is provided in the proof of Lemma \ref{lem:remove_p-r_then_solvable}). In addition, if there is an edge $\edge{3}{17}$ in $\graphcomp(G)$, we direct it by $17 \rightarrow 3$. 
\end{definition}

By the proof of Lemma \ref{lem:remove_p-r_then_solvable}, if the edges $\edge{2}{3}$ and/or $\edge{2}{17}$ are in $\graphcomp(G)$, the orientations must be $3 \leftarrow 2 \rightarrow 17$. The choice of $17 \rightarrow 3$ over $3 \rightarrow 17$ is somewhat arbitrary, but it helps for consistency in the proof of Theorem \ref{thm:PSL217_classification}.

If $G$ is a pseudo $\PSL(2, 17)$-solvable group, then Lemma \ref{lem:psl217_no_3-17} shows that $\{2, 3, 17\}$ and $\{3, 17, p\}$ (for some $p \notin \{2, 3, 17\}$) are the only possible triangles that can exist in $\graphcomp(G)$. With $\PSL(2, 7)$ we were able to rule out $\{2, 7, p\}$ triangles and with $A_6$, we were able to rule out $\{3, 5, p\}$ triangles. Likewise, for $\PSL(2, 17)$, we will be able to rule out $\{3, 17, p\}$ triangles, but it will take considerably more work than in the previous cases. 


Our strategy will be to reduce to the cases of either $G = N.\PSL(2,17)$ or $G = N.\SL(2,17)$, where $2, 3, 17 \nmid |N|$. This would allow us to apply techniques which require a coprime action.
We spend the next few lemmas reducing to these cases. First, we prove an elementary lemma about generalized quaternion groups, which we apply to $Q_{32}$, the Sylow $2$-subgroup of $\SL(2, 17)$. We then generalize this lemma to all of $\SL(2, 17)$. 

\begin{lemma}\label{lem:Q32}
Suppose we have a group $G$ with two distinct normal subgroups $N_1, N_2 \unlhd G$ such that $N_1 \cong N_2 \cong C_2$, $N_1N_2 \cong V_4$, and $G / N_1 \cong G / N_2 \cong Q_{2^k}$ where $k \geq 3$. Then $G \cong N_1 \times Q_{2^k}$.
\end{lemma}

\begin{proof}
Let $n_1,n_2$ be the generators of $N_1,N_2$ respectively. Let $N=N_1N_2 \cong V_4$. Then $G/N \cong (G/N_1)/(N/N_1) \cong Q_{2^k}/C_2 \cong D_{2^{k-1}}$. Select a subgroup $M$ where $N \unlhd M \leq G$ and $M/N$ is isomorphic to a cyclic subgroup of order $2^{k-2}$ in $G/N$. (This subgroup is unique for $k \geq 4$.) Choose an element $a \in G$ such that $M/N= \langle aN \rangle$.
Note that $M/N_1 \leq G/N_1 \cong Q_{2^k}$. Thus $M/N_1$ is cyclic or generalized quaternion. Suppose that $M/N_1$ is generalized quaternion. Then $M/N \cong (M/N_1)/(N/N_1)$ and $N/N_1 \cong C_2$. The quotient of any quaternion group by the cyclic subgroup of order $2$ is dihedral, but $M/N$ is cyclic and thus not dihedral. Thus we reach a contradiction, so $M/N_1$ must be the cyclic subgroup of order $2^{k-1}$.

Therefore, there exists an element $m \in N$ such that $M/N_1= \langle (am)N_1 \rangle$. The element order of $(am)N_1$ in $G/N_1$ is $2^{k-1}$. But $mN_1 \in G/N_1$ is a central involution, so this implies the element order of $aN_1$ in $G/N_1$ is $2^{k-1}$. Recall that $a^{2^{k-2}} \in N$ because $aN$ generates $M/N$, but $a^{2^{k-2}} \notin N_1$. By an analogous argument, $a^{2^{k-2}} \in N \setminus N_2$. Thus $a^{2^{k-2}}=n_1n_2$.


Now choose $b \in G$ such that $bN \in (G/N) \setminus (M/N)$. Since $G/N \cong D_{2^{k-1}}$, $b^2 \in N$. Also note that $bN_1 \in (G/N_1) \setminus (M/N_1)$, so the order of $bN_1$ is $4$ because all elements outside the index $2$ cyclic subgroup of a quaternion group are of order $4$. Thus $b^2 \notin N_1$, so $b^2 \in N \setminus N_1$. By an analogous argument, $b^2 \in N \setminus N_2$. Thus $b^2=n_1n_2=a^{2^{k-2}}$, which implies that $|a|=2^{k-1}$ and $|b|=4$.

Furthermore, since $bN_1 \in (G/N_1) \setminus (M/N_1)$, it is true that $(aN_1)^{(bN_1)}=(aN_1)^{-1}$ and therefore $a(a^b) \in N_1$. By the same argument, $a(a^b) \in N_2$. Therefore $a(a^b)$ is trivial and $a^b=a^{-1}$. Thus the subgroup $U= \langle a,b \rangle \leq G$ is isomorphic to $Q_{2^k}$. Also $n_1n_2 \in U$ is the unique involution, so $n_1 \notin U$. Therefore $G=U \times N_1 \cong Q_{2^k} \times N_1$.
\end{proof}

\begin{lemma}\label{lem:SL(2, 17)_two_legs}
Suppose we have a group $G$ with two distinct normal subgroups $N_1, N_2 \unlhd G$ such that $N_1 \cong N_2 \cong C_2$, $N_1N_2 \cong V_4$, and $G / N_1 \cong G / N_2 \cong \SL(2, 17)$. Then $G \cong C_2 \times \SL(2, 17)$. 
\end{lemma}

\begin{proof}
Abbreviate $\SL(2, 17)$ by $S$. Let $\phi: G \to S$ be the projection from $G$ onto $S$ whose kernel is $N_1$. Then $G = \phi^{-1}(S)$. The group $S$ has a Sylow $2$-subgroup $P_2$ isomorphic to $Q_{32}$, a Sylow $3$-subgroup $P_3$ isomorphic to $C_9$, and a Sylow $17$-subgroup $P_{17}$ isomorphic to $C_{17}$. Now we can consider the subgroups $\phi^{-1}(P_2), \phi^{-1}(P_3), \phi^{-1}(P_{17}) \leq G$. Each of these has $N_1$ as a normal subgroup, and each is an extension of $N_1$. That is, $\phi^{-1}(P_i) \cong N_1. P_i$ for $i \in \{2, 3, 17\}$. For $i = 3$, the extension is coprime and so it splits by Schur-Zassenhaus. Every group must act trivially on $N_1 \cong C_2$, so $\phi^{-1}(P_3) = N_1 \times T_3$, where $T_3 \leq \phi^{-1}(P_3)$ is isomorphic to $P_3$, and the direct product is internal. Similarly, $\phi^{-1}(P_{17}) = N_1 \times T_{17}$, where $T_{17} \leq \phi^{-1}(P_{17})$ is isomorphic to $P_{17}$. 

Now we handle the $i=2$ case. Note that $\SL(2, 17)$ has a unique involution $\alpha$. Since $N_2$ is not the kernel of $\phi$, the nontrivial element of $N_2$ must be mapped to $\alpha$. Obviously, $\alpha \in P_2$, and so $N_2 \leq \phi^{-1}(P_2)$. Note that $\phi^{-1}(P_2)$ has the form $N_1.Q_{32}$, and so it must be a Sylow $2$-subgroup of $G$ by considering its order. Since Sylow subgroups are mapped to Sylow subgroups under surjective homomorphisms, then $\psi(\phi^{-1}(P_2))$ is a Sylow $2$-subgroup of $S$, where $\psi$ is the projection of $G$ onto $S$ whose kernel is $N_2$. So $\psi(\phi^{-1}(P_2)) \cong P_2$. Thus, $\phi^{-1}(P_2) / N_1 \cong P_2 \cong \phi^{-1}(P_2) / N_2 \cong Q_{32}$, so we may apply Lemma \ref{lem:Q32} to get that $\phi^{-1}(P_2) = N_1 \times T_2$, where $T_2 \leq \phi^{-1}(P_2)$, $T_2 \cong P_2$, and the direct product is internal. 

Now, we have $\phi^{-1}(P_i) = N_1 \times T_i$, where $T_i \leq G$ and $T_i \cong P_i$, for all $i \in \{2, 3, 17\}$. So consider the subgroup $T = N_1 \times \langle T_2, T_3, T_{17} \rangle \leq G$, where the direct product is internal. The group $\langle T_2, T_3, T_{17} \rangle$ has subgroups $T_2$, $T_3$, and $T_{17}$, and so has order at least $2^5 \cdot 3^2 \cdot 17$. Then $|T| \geq 2^6 \cdot 3^2 \cdot 17$. However, $|T| \leq |G| = 2^6 \cdot 3^2 \cdot 17$, and so $|T| = |G|$. Thus, $T = G$. 

If we let $H = \langle T_2, T_3, T_{17} \rangle$, then $H$ is a complement to $N_1$ in $G$, since $G = N_1 \times H$. Since $N_1 \unlhd G$, then $G / N_1 \cong H$. However, by assumption, $G / N_1 \cong \SL(2, 17)$. Therefore, $G = N_1 \times H \cong C_2 \times \SL(2, 17)$ as desired.
\end{proof}

\begin{lemma}\label{lem:psl217_or_sl217_subgroup}
Let $G = P.\PSL(2, 17)$ where $P$ is a $2$-group. Then either $G$ contains a subgroup isomorphic to $\PSL(2, 17)$ or it contains a subgroup isomorphic to $\SL(2, 17)$. 
\end{lemma}
\begin{proof}
Take the chief series $G = P_1.P_2.\cdots.P_k.\PSL(2, 17)$. Then each chief factor $P_i$ is an elementary abelian $2$-group. We induct on the number $k$ of abelian chief factors of $G$. For the $k=0$ case, $G = \PSL(2, 17)$, so we are done. 

Next, assume the proposition holds for $k \geq 0$, and assume $G$ has $k+1$ abelian chief factors. So $G = P_1.P_2.\cdots.P_k.P_{k+1}.\PSL(2, 17)$. Then $G$ has a quotient $K = P_2.\cdots.P_k.P_{k+1}.\PSL(2, 17)$, which has $k$ abelian chief factors. By the inductive hypothesis, $K$ contains a subgroup isomorphic to either $\PSL(2, 17)$ or $\SL(2, 17)$. 

Suppose first that $\PSL(2, 17) \leq K$. Then $G = P_1.K$ and so $G$ has a subgroup $P_1.\PSL(2, 17)$. If the extension splits, $\PSL(2, 17) \leq P_1.\PSL(2, 17) \leq G$ and we are done. If the extension does not split, then $\PSL(2, 17)$ must act irreducibly on $P_1$ since $P_1$ is a chief factor. By \cite[Theorem 3]{Burichenko-Odd-q}, this extension must be equal to $\SL(2, 17)$ and we are done.

Suppose next that $\SL(2, 17) \leq K$. Then $G = P_1.K$ and so $G$ has a subgroup $L = P_1.\SL(2, 17)$. Let $Z = Z(\SL(2, 17)) \cong C_2$. Then $L = P_1.Z.\PSL(2, 17)$. Now, viewing $P_1$ as an $\mathbb{F}_2 Z$-module, $P_1$ is irreducible. Then by \cite[V,, Satz 5.17]{EndlicheGruppenI}, the factors $Z$ and $P_1$ commute, so $L = Z.P_1.\PSL(2, 17)$. If the extension $P_1.\PSL(2, 17)$ splits, then $L$ has a subgroup $Z . \PSL(2, 17) = SL(2, 17)$ and we are done. If not, then we again cite \cite[Theorem 3]{Burichenko-Odd-q} to get that $P_1.\PSL(2, 17) = \SL(2, 17)$. So $P_1 \cong C_2 \cong Z$, and $L / P_1 \cong L / Z \cong \SL(2, 17)$. Then by Lemma \ref{lem:SL(2, 17)_two_legs}, $L \cong C_2 \times \SL(2, 17)$. Thus, $\SL(2, 17) \leq L \leq G$ and we are done.
\end{proof}

Lemma \ref{lem:psl217_or_sl217_subgroup} allows us to reduce extensions of $\PSL(2, 17)$ by a $2$-group to either $\PSL(2, 17)$ or $\SL(2, 17)$. Next, we use this to show that if $\{2,3,17\}$ forms a triangle in $\graphcomp(G)$, then there can be no other edges of the form $\edge{2}{p}$ or $\edge{3}{p}$. At the same time, we rule out the possibility of a $\{3, 17, p\}$ triangle.

\begin{lemma}\label{lem:psl217-no2p}
If $G$ is pseudo $\PSL(2, 17)$-solvable and $\{2, 3, 17\}$ forms a triangle in $\graphcomp(G)$, then there is no edge $\edge{2}{p}$, for $p \notin \{2, 3, 17\}$. 
\end{lemma}

\begin{proof}
Suppose for contradiction that $\edge{2}{p}$ is an edge in $\graphcomp(G)$. Since $\graphcomp(G)$ contains a triangle, then $G$ must be strictly pseudo $\PSL(2, 17)$-solvable. By Corollary \ref{cor:subgroupcor}, $G$ has a subgroup $K = N.\PSL(2, 17)$ where $N$ is solvable and $\pi(K) = \pi(G)$. Thus, there is a $\{2, 3, 17\}$ triangle in $\graphcomp(K)$, and $\edge{2}{p} \in \graphcomp(K)$.

If $3 \mid N$, then the subgroup $K_1 = N.D_{34} \leq K$ is solvable, and $\pi(K_1) = \pi(K)$. Then $\graphcomp(K_1)$ is triangle-free, and since $\pi(K_1) = \pi(K)$, so is $\graphcomp(K)$. This contradicts $\graphcomp(K)$ having a triangle, so $3 \nmid |N|$. Similarly, if $17 \mid |N|$, then considering the subgroup $K_2 = N.D_{18} \leq K$, we arrive at a contradiction. Therefore, $17 \nmid |N|$.

Let $\pi := \{2, 3, 17, p\}$. Since $K$ is $\pi$-separable, we may consider a Hall $\pi$-subgroup $H_\pi$ of $K$, which by Lemma \ref{lem:Hallsubgroups}, takes the form $H_\pi = H_{2p}.\PSL(2, 17)$, where $H_{2p}$ is a Hall $\{2, p\}$-subgroup of $N$. Consider a chief series of $H_\pi = H_{2p}.\PSL(2, 17)$, so that the chief factors of $H_{2p}$ are elementary abelian $2$-groups or elementary abelian $p$-groups. From this, we observe that $H_\pi$ must have a quotient group $H$ of the form $H = P.S.\PSL(2,17)$, where $P$ is a nontrivial $p$-group, and $S$ is a (possibly trivial) $2$-group. 
Apply Lemma \ref{lem:psl217_or_sl217_subgroup} to get that $H$ has a subgroup $L$ which takes the form $L = P.\PSL(2, 17)$ or $L = P.\SL(2, 17)$. Since $\pi(L) = \pi$ and $L$ is a section of $K$, $\graphcomp(L)$ contains the edge $\edge{2}{p}$ and the triangle $\{2, 3, 17\}$.

If $L = P.\SL(2,17)$, we arrive at a contradiction since $\SL(2, 17)$ has an element of order $6$, which would contradict the $\{2, 3, 7\}$ triangle in $\graphcomp(L)$. Thus, $L = P.\PSL(2,17)$. So consider the subgroup $P.D_{16} \leq L$, which must have an element of order $2p$ since $D_{16}$ cannot act Frobeniusly and $P$ is a $p$-group. Therefore, $L$ has an element of order $2p$, which contradicts the $\edge{2}{p}$ edge in $\graphcomp(L)$. Hence, there is no edge $\edge{2}{p}$ in $\graphcomp(G)$.
\end{proof}

\begin{lemma}\label{lem:SL(2,17)reptheory}
There exists an element $t \in \SL(2,17)$ of order $3$ such that for every irreducible complex representation $\rho$ of $\SL(2,17)$, $\rho(t)$ has a fixed point.
\end{lemma}

\begin{proof}
Let $\rho : \SL(2,17) \rightarrow \GL(n, \mathbb{C})$ be an irreducible complex representation. Note that in $\SL(2,17)$, all elements of order $3$ are conjugate. Thus, it will suffice to find any element $t$ of order $3$ for which $\rho(t)$ has a fixed point.

If $\rho$ is non-faithful, then the kernel of $\rho$ contains $Z(\SL(2,17))$, so $\rho$ comes from a representation $\sigma$ of $\SL(2,17) / Z(\SL(2,17)) \cong \PSL(2,17)$. Now suppose for contradiction that no element of order $3$ in $\PSL(2,17)$ has a fixed point under $\sigma$. Let $r$ be a prime such that $r \nmid |\PSL(2,17)|$. Then applying \cite[Lemma 3.4]{A5classification}, we find a nontrivial module $R$ of $\PSL(2,17)$ over a finite field of characteristic $r$ such that no element of order $3$ in $\PSL(2,17)$ has a fixed point. Consider the group $G = R \rtimes \PSL(2,17)$, where the semidirect product comes from the module action. By Lemma \ref{lem:extensions_of_A6}, we must have $R = 1$, a contradiction since $R$ is nontrivial. Therefore, there must be an element $s \in \PSL(2,17)$ of order $3$ for which $\sigma(s)$ has a fixed point. Then there is a corresponding element $t \in \SL(2,17)$ of order $3$ such that $\rho(t)$ has a fixed point as well. This completes the non-faithful case.

Now suppose that $\rho$ is a faithful representation of $\SL(2,17)$. These representations are well studied, and are usually divided into the \textit{discrete series} and the \textit{principal series} as in \cite{SilbergerSL217}. First take the discrete series. The Sylow $17$-subgroup always acts Frobeniusly in these representations, as noted in \cite[Section 2]{SilbergerSL217}. Also note that the Sylow $2$-subgroup is $Q_{32}$, which must act Frobeniusly in any faithful representation. 
If the Sylow $3$-subgroup also acts without fixed points, then the whole group, $\SL(2,17)$, would act Frobeniusly. This is a contradiction, because $\SL(2,17)$ is perfect and $\SL(2,5)$ is the only perfect Frobenius complement, as stated in \cite[Section 6A]{Isaacs2008FiniteGT}. Therefore the Sylow $3$-subgroup does not act Frobeniusly and the elements of order $3$ act with fixed points.

Now consider the principal series. Table 5 of \cite{HumprheysSL217} lists the character table for $\SL(2,17)$ with the top five rows comprising the principal series. Note that a representation is faithful if and only if the central involution, $z$, of $\SL(2,17)$ is not in the kernel, and $z$ is in the kernel of a representation exactly when $\chi(z)=\chi(1)$ where $\chi$ is the character corresponding to $\rho$. In \cite[Table 5]{HumprheysSL217}, it is easy to see that in the principal series, $\chi(z)=\chi(1)$ except in row three when $i=1,3,5,7$. That is, the principal series contains only four faithful representations, all of degree $18$. However, as explained in \cite[Section 7]{HumprheysSL217}, all degree $18$ representations are induced from representations of a subgroup of order $16 \cdot 17$. It is clear that all elements of order $3$ have a fixed point in the induced module. Thus the elements of order $3$ act with fixed points in both the principal and discrete series.

Therefore there exists an element of order $3$ which acts with fixed points in every irreducible complex representation.
\end{proof}

The authors note that we also verified the above lemma computationally using $\text{GAP}$.

\begin{lemma}\label{lem:psl217_no_3-p_or_3-col}
Let $G$ be a strictly pseudo $\PSL(2, 17)$-solvable group and $p \notin \{2, 3, 17\}$ a prime divisor of $|G|$. Then at least one of the following holds:
\begin{enumerate}
    \item $\graphcomp(G)$ is $3$-colorable and triangle-free.
    \item $\edge{3}{p}$ is not an edge in $\graphcomp(G)$. 
\end{enumerate} 
\end{lemma}

\begin{proof}
By Corollary \ref{cor:subgroupcor}, $G$ has a subgroup $K = N.\PSL(2, 17)$, where $N$ is solvable and $\pi(G) = \pi(K)$. If $3 \mid |N|$, consider $K_1 = N.D_{34} \leq K \leq G$. Since $K_1$ is solvable, $\graphcomp(K_1)$ is $3$-colorable and triangle-free, and since $\pi(K_1) = \pi(G)$, then $\graphcomp(G)$ is $3$-colorable and triangle-free. Then we are done, and so we may assume $3 \nmid |N|$. Similarly, if $17 \mid |N|$, consider $K_2 = N.A_4 \leq K \leq G$. Since $K_2$ is solvable, $\graphcomp(K_2)$ is $3$-colorable and triangle-free, so as before, then $\graphcomp(G)$ is $3$-colorable and triangle-free. Then we are done, and so we may assume $17 \nmid |N|$.

So $3, 17 \nmid |N|$. Let $\pi = \{2, 3, 17, p\}$, so $K$ is $\pi$-separable and contains a Hall $\pi$-subgroup $H_\pi$. By Lemma \ref{lem:Hallsubgroups}, $H_\pi$ takes the form $H_\pi = H_{2p}.\PSL(2, 17)$, where $H_{2p}$ is the Hall $\{2, p\}$-subgroup of $N$. By considering a chief series of $H_\pi$, we have that $H_\pi$ has a quotient $L = P.Q.\PSL(2, 17)$, where $P$ is a nontrivial $p$-group and $Q$ is a possibly trivial $2$-group. By Lemma \ref{lem:psl217_or_sl217_subgroup}, $L$ has a subgroup $M$ which either takes the form $M = P.\PSL(2, 17)$ or $M = P.\SL(2, 17)$. If the former holds, apply Lemma \ref{lem:extensions_of_A6} to retrieve an element of $3p$ in $M$. If the latter holds, apply Lemma \ref{lem:rep_theory_biz} with Lemma \ref{lem:SL(2,17)reptheory} to retrieve an element of $3p$ in $M$. In either case, $M \leq L$ has an element of order $3p$. Since $L$ is a quotient of $H_\pi \leq G$, then $G$ has an element of order $3p$ as well, and so $\edge{3}{p} \notin \graphcomp(G)$. 
\end{proof}

\begin{corollary}\label{cor:psl217-only1triangle}
Let $G$ be pseudo $\PSL(2, 17)$-solvable. If a triangle exists in $\graphcomp(G)$, then it is the triangle $\{2, 3, 17\}$.
\end{corollary}

\begin{proof}
By Lemma \ref{lem:psl217_no_3-17}, the only possible triangles in $\graphcomp(G)$ are $\{2, 3, 17\}$ and $\{3, 17, p\}$, for $p \notin \{2, 3, 17\}$. But by Lemma \ref{lem:psl217_no_3-p_or_3-col}, there are no $\{3, 17, p\}$ triangles.
\end{proof}

\begin{corollary}\label{cor:psl217-no3p}
Let $G$ be pseudo $\PSL(2, 17)$-solvable. If $\{2, 3, 17\}$ forms a triangle in $\graphcomp(G)$, then $\edge{3}{p}$ is not an edge in $\graphcomp(G)$. 
\end{corollary}

\begin{corollary}\label{cor:psl217-3col}
Let $G$ be pseudo $\PSL(2, 17)$-solvable. Then $\graphcomp(G)$ is $3$-colorable.
\end{corollary}

\begin{proof}
By Lemma \ref{lem:psl217_no_3-p_or_3-col}, $\graphcomp(G)$ is either $3$-colorable and triangle-free or it has no $\edge{3}{p}$ edge for each $p \notin \{2, 3, 17\}$. If $\graphcomp(G)$ is $3$-colorable and triangle free, we are done. Otherwise, the only vertices $3$ can be connected to are $2$ and $17$. Thus, the vertex $3$ has degree at most $2$. So in order to show that $\graphcomp(G)$ is $3$-colorable, it suffices to show that $\graphcomp(G) \setminus \{3\}$ is $3$-colorable. However, $\graphcomp(G) \setminus \{3\}$ is a subgraph of $\graphcomp(G) \setminus \{\edge{3}{p}\}$, which is $3$-colorable by Lemma \ref{lem:psl217_no_3-17}. A subgraph of a $3$-colorable graph is $3$-colorable, so we are done.
\end{proof}

Finally, we are ready to classify the prime graphs of pseudo $\PSL(2,17)$-solvable groups. The classification turns out to be the same as for $\PSL(2,7)$, so we can follow a similar proof pattern as in Theorem \ref{thm:PSL_classification}.

\begin{theorem}\label{thm:PSL217_classification}
Let $\Gamma$ be a simple graph. Then $\Gamma$ is isomorphic to the prime graph of a pseudo $\PSL(2, 17)$-solvable group $G$ if and only if one of the following holds:
\begin{enumerate}
    \item $\graphcomp$ is triangle-free and $3$-colorable.
    \item $\graphcomp$ has one exactly one triangle $\{a,b,c\}$, the vertices $a, b$ are not connected to any other vertices in $\graphcomp$, and $\graphcomp$ has a $3$-coloring for which all the neighbors of $c$ other than $a$ and $b$ have the same color.
\end{enumerate}
\end{theorem}

\begin{proof}
We first prove the forward direction. Assume that $\Gamma$ is the prime graph of a pseudo $\PSL(2, 17)$-solvable group $G$. If $\graphcomp$ is triangle-free, then by Corollary \ref{cor:psl217-3col}, $\graphcomp$ is $3$-colorable and we get (1). Otherwise, by Corollaries \ref{cor:psl217-only1triangle} and \ref{cor:psl217-no3p} and Lemma \ref{lem:psl217-no2p}, $\graphcomp$ has exactly one triangle, which is $\{2,3,17\}$, and $\graphcomp$ contains no edges of the form $\edge{2}{p}$ or $\edge{3}{p}$ for any $p \notin \{2,3,17\}$. 

Examine the Frobenius digraph $\digraph(G)$. Note that there will be no $3$-paths in this case by the proof of Theorem \ref{lem:psl217_no_3-17}. Now, if $7$ is adjacent to some $p \notin \{2,3,17\}$ in $\digraph(G)$, we have that the orientation must be $2 \rightarrow 17 \rightarrow p$. Therefore, if $q$ is a prime adjacent to $p$, then the orientation of the $p-q$ edge must be $q \to p$. Now define a coloring on this digraph. Label all vertices with zero out-degree with $\mathcal{I}$, all vertices with zero in-degree and nonzero out-degree with $\mathcal{O}$, and all vertices with non-zero in- and out-degree with $\mathcal{D}$. This is a well-defined three coloring, since $\digraph(G)$ contains no $3$-paths. And all vertices adjacent to $17$ other than $2$ and $3$ are labeled with $\mathcal{I}$, so we arrive at (2). This concludes the forward direction of the proof.

We next prove the backwards direction. Consider the irreducible complex representation $\rho_2$ of $\PSL(2,17)$ listed in Table \ref{table:3} for which $g \in \PSL(2, 17)$ acts fixed point freely if and only if $|g| = 17$. This is similar to the irreducible representation $\rho_1$ of $\PSL(2,7)$ also listed in Table \ref{table:3}. Due to this similarity, we may pursue the same construction as in Theorem \ref{thm:PSL_classification}, only replacing every occurrence of $7$ with the number $17$, and replacing every occurrence of the irreducible representation $\rho_1$ with the irreducible representation $\rho_2$. (Also, the colors of the vertices $2$ and $3$ in the construction will switch.) Since the constraints on $\Gamma$ are the same between this theorem and Theorem \ref{thm:PSL_classification}, we may construct a pseudo $\PSL(2, 17)$-solvable group $G$ such that $\graphcomp(G) \cong \graphcomp$.
\end{proof}


\section{Pseudo $\mathcal{K}_3$-solvable Groups}\label{section:multiple-Ts}

Now we explore groups with multiple non-isomorphic composition factors which are $K_3$-groups. Recall that the $K_3$-groups are the nonabelian simple groups whose orders have exactly three prime factors and that $\mathcal{K}_3$ denotes the set of all $K_3$-groups. Heuristically, we observe that adding $K_3$-groups as composition factors to a group should tend to simplify the complement to its prime graph. This is because all eight $K_3$-groups share the prime divisors $2$ and $3$, and we expect that all $K_3$-groups act trivially on one another. This removes many edges from the complement of the prime graph. Proposition \ref{prop:multiple-ts} states a result which agrees with this intuition. 

To prove that proposition, we first need to prove the following generalization to Corollary \ref{cor:subgroupcor}. 

\begin{lemma}\label{lem:strongsubgrouplemma}
Let $G$ be strictly pseudo $\{T_1, \dots, T_k\}$-solvable, where $1 \leq k \leq 8$, and each $T_i$ is a $K_3$-group. Suppose $G$ has a chief series for which the lowest nonabelian chief factor is $T_1^n$. Then $G$ has a subgroup $K = N.T_1$, where $N$ is solvable and $\pi(G) = \pi(K)$. In particular, $\graphcomp(G)$ is obtained by removing edges from $\graphcomp(K)$.
\end{lemma}

\begin{proof}
Let $m$ be the number of nonabelian chief factors in the given chief series for $G$. We will induct on $m$. If $m = 1$, then $G$ is strictly pseudo $T_1$-solvable, so we are done by Corollary \ref{cor:subgroupcor}. \\

Now suppose the statement is true for some $m \geq 1$ (and any $k$), and let $G$ be strictly pseudo $\{T_1, \dots, T_k\}$-solvable with a chief series with $m + 1$ nonabelian chief factors, the lowest of which is $T_1^n$. Suppose the top nonabelian chief factor is $T_i^\ell$. Thus, we can write $G = N_1.T_i^\ell.H$, where $H$ is solvable and $N_1$ is pseudo $\{T_1, \dots, T_k\}$-solvable. By Corollary \ref{cor:subgroupcor}, $T_i^\ell.H$ has a subgroup $N_2.T_i$, where $N_2$ is solvable, such that $\pi(T_i^\ell.H) = \pi(N_2.T_i)$. If $\pi(T_i) = \{2,3,p\}$, then let $P$ be a Sylow $p$-subgroup of $T_i$. Then, $T_i^\ell.H$ has a subgroup $N_2.P$, where $\pi(T_i^\ell.H) = \pi(N_2.P) \cup \{2,3\}$. But since $T_1^n$ is a chief factor of $N_1$, we know $2, 3 \mid |N_1|$. Thus, $G$ has a subgroup $K_1 = N_1.N_2.P$ such that $\pi(K_1) = \pi(G)$. And, $K_1$ has a chief series with $m$ nonabelian chief factors, the lowest of which is $T_1^n$. By the inductive hypothesis, $K_1$ has a subgroup $K = N.T_1$ for which $N$ is solvable and $\pi(K) = \pi(K_1) = \pi(G)$.
\end{proof}

\begin{proposition}\label{prop:multiple-ts}
Let $G$ be a pseudo $\mathcal{K}_3$-solvable group with at least $2$ (not necessarily distinct) nonabelian composition factors. Then $\graphcomp(G)$ is triangle-free and $3$-colorable.
\end{proposition}

\begin{proof}
Take a chief series of $G$, and let $T_1^m$ be the lowest nonabelian chief factor. Since $T_1 \in \mathcal{K}_3$, we know $\pi(T_1) = \{2,3,p\}$ for some prime $p$. We will first show that $\graphcomp(G)[\{2,3,p\}]$ is disconnected. 

If $G$ is strictly pseudo $T_1$-solvable, then by assumption, $G$ has at least two copies of $T_1$ in its composition series. Then by the contrapositive of Lemma \ref{lem:onecopyofT}, $\graphcomp(G)[\{2,3,p\}]$ is indeed disconnected. 

Thus, we may assume $G$ has another nonabelian composition factor which is not isomorphic to $T_1$. Then, $G$ contains a section $H = T_1^n.S.T_2^k$, where $S$ is solvable, $T_1 \neq T_2 \in \mathcal{K}_3$, and $n, k \in \mathbb{N}$. If $n > 1$, then $H$ contains a subgroup $T_1^n$ whose prime graph complement is disconnected, so $\graphcomp(G)[\{2,3,p\}]$ is disconnected. Therefore, we may assume $n = 1$, so $H = T_1.S.T_2^k$. We further restrict to the subgroup $H_1 = T_1.S.T_2$. 

Consider the conjugation action of $H_1$ on $T_1$. Let $\phi : H_1 \rightarrow \Aut(T_1)$ denote the homomorphism of this action. Let $C = C_{H_1}(T_1)$ be the kernel of $\phi$. By the first isomorphism theorem, $H_1 / C \cong \phi(H_1)$. Considering orders, $|C| = \frac{|H_1|}{|\phi(H_1)|}$. Noting that $\phi(H_1)$ at least contains $\Inn(T_1) = T_1$, we get that $|\phi(H_1)| = a|T_1|$ for some $a$ dividing $|\Out(T_1)|$. Thus, $|C| = \frac{|T_1||S||T_2|}{a|T_1|} = \frac{|T_2||S|}{a}$. Note that by referring to Table \ref{table:1}, we find that $a = 1, 2, 3$, or $4$.

Suppose $\pi(T_2) = \{2,3,q\}$. We take cases on whether $p$ and $q$ are distinct.

If $p = q$, then since $|C| = \frac{|T_2||S|}{a}$ and $a = 1, 2, 3$, or $4$, we will either get $2, p \mid |C|$ or $3, p \mid |C|$. Thus, $C$ contains an element of order $p$ and an element of order $2$ or $3$, both of which commute with elements of order $2, 3$, and $p$ in $T_1$. Therefore, $H_1$ contains elements of order $6, 2p$, and $3p$, and $\graphcomp(G)[\{2,3,p\}]$ is disconnected.

If $p \neq q$, let $Q$ be a Sylow $q$-subgroup of $H_1$. Then since $q \nmid a$ and $|C| = \frac{|T_2||S|}{a}$, we know $Q \leq C$. For convenience, let $M = T_1.S$, so $M \unlhd H_1$. Since $Q \leq C \unlhd H_1$, we obtain $1 < QM/M \leq CM/M \unlhd H_1 / M \cong T_2$. Since $CM/M$ is a nontrivial normal subgroup of $H_1 / M$, which is simple, we must have $CM/M = H_1 / M$ and hence $CM = H_1$. 

Therefore, $T_2 \cong CM/M = \{cM \mid c \in C\}$. Since $T_2$ has elements of order $2$ and $3$, this implies the existence of elements $c_2, c_3 \in C$ for which the order of $c_2 M$ is $2$ and the order of $c_3 M$ is $3$. Then $2$ divides the order of $c_2$ and $3$ divides the order of $c_3$, and $c_2, c_3$ commute with elements of order $2, 3$, and $p$ in $T_1$. Again, this implies $\graphcomp(G)[\{2,3,p\}]$ is disconnected.

We have proven that in all cases, $\graphcomp(G)[\{2,3,p\}]$ is disconnected. By Lemma \ref{lem:strongsubgrouplemma}, $G$ contains a pseudo $T_1$-solvable subgroup $K$ such that $\pi(K) = \pi(G)$. Referring to Table \ref{table:1}, $T_1$ satisfies the conditions of Lemma \ref{lem:remove_p-r_then_solvable}. So $\graphcomp(K)$ minus some edge between $2$, $3$, and $p$ is $3$-colorable and triangle-free. But we just showed there are no edges between $2$, $3$, and $p$ in $\graphcomp(K)$. Therefore, $\graphcomp(K)$ and thus $\graphcomp(G)$ is triangle-free and $3$-colorable.
\end{proof}

As a consequence of Proposition \ref{prop:multiple-ts}, we obtain a large class of groups for which the prime graphs are isomorphic to the prime graph of a solvable group. The following theorem summarizes this idea.

\begin{theorem}\label{thm:lookslikesolvable}
Let $G$ be a pseudo $\mathcal{K}_3$-solvable group, and suppose one of the following holds:
\begin{enumerate}
    \item $G$ has no nonabelian composition factors (i.e, $G$ is solvable);
    \item $G$ has exactly one nonabelian composition factor, which is $\PSL(3,3), U_3(3)$, or $U_4(2)$; or
    \item $G$ has at least two (not necessarily distinct) nonabelian composition factors.
\end{enumerate}
Then, $\graphcomp(G)$ is triangle-free and $3$-colorable.
\end{theorem}

\begin{proof}
This follows from Lemma \ref{lem:solvablegroups}, Theorem \ref{thm:U33_classification}, Theorem \ref{thm:U42_classification}, Theorem \ref{thm:psl33-classification}, and Proposition \ref{prop:multiple-ts}.
\end{proof}

\section{Outlook}\label{section:outlook}

    
    
    
    

Several natural questions arise out of the classifications of this paper. The classification of prime graphs of \textit{strictly} pseudo $T$-solvable groups remains an open problem for each $T \in \mathcal{K}_3$, as discussed in Remark \ref{rmk:Tim's-remark}. Topics for further research include the study of pseudo $T$-solvable groups where $T$ is allowed to be a $K_4$-group (a simple nonabelian group with exactly $4$ prime divisors), an alternating group, or a Suzuki group. Allowing Suzuki groups into the composition series could be particularly interesting because this could lead to a classification of prime graphs for all $3'$-groups. \\
Observe that as we moved away from solvability in this paper, the complements of the prime graphs of all
the groups we studied here stayed $3$-colorable, but some lost triangle-freeness (that is, have a triangle). This raises the following question posed by Maslova in \cite[Section 5]{inproceedings}: Can one also deviate from solvability ``the other way round'', that is, keep triangle-freeness, but violate 3-colorablility? The smallest triangle-free graph which is not 3-colorable is known as the Groetzsch graph. Originally, we hoped to find a group with prime graph complement being the Groetzsch graph in the realm of pseudo $\mathcal{K}_3$-solvable groups, but as we have seen, such a group does
not exist. We now pass on the challenge to the reader.


\section{Appendix}

\begin{table}[h!]
\begin{tabular}{|c||c|c|c|c|c|}
    \hline
     $G$ & $|G|$ & $\Out(G)$ & Sylow Subgroups & Relevant Subgroups & $\graphcomp(G)$ \\ \hline \hline
     $A_5$ & $2^2 \cdot 3 \cdot 5$ & $C_2$ & $V_4$, $C_3$, $C_5$ & $D_{10}$, $A_4$, $S_3$ & \primetri\\
     $\PSL(2, 7)$ & $2^3 \cdot 3 \cdot 7$ & $C_2$ & $D_8$, $C_3$, $C_7$ & $S_4$, $A_4$, $C_7 \rtimes C_3$ & \primetri \\
     $A_6$ & $2^3 \cdot 3^2 \cdot 5$ & $V_4$ & $D_8$, $C_3^2$, $C_5$ & $C_3^2 \rtimes C_4$, $S_4$, $D_{10}$ $A_5$ & \primetri \\
     $\PSL(2, 8)$ & $2^3 \cdot 3^2 \cdot 7$ & $C_3$ & $C_2^2$, $C_9$, $C_7$ & $C_2^3 \rtimes C_7$, $D_{14}$, $D_{18}$, $S_3$ & \primetri\\
     $\PSL(2, 17)$ & $2^4 \cdot 3^2 \cdot 17$ & $C_2$ & $D_{16}$, $C_9$, $C_{17}$ & $S_4$, $D_{18}$,  $D_{34}$ & \primetri\\
     $\PSL(3, 3)$ & $2^4 \cdot 3^3 \cdot 13$ & $C_2$ & $SD_{16}$, $C_3^2 \rtimes C_3$, $C_{13}$ & $C_{13} \rtimes C_3$ & \primeline\\
     $U_3(3)$ & $2^5 \cdot 3^3 \cdot 7$ & $C_2$ & $C_4^2 \rtimes C_2$, $C_3^2 \rtimes C_3$, $C_7$ & $\PSL(2, 7)$ & \primeline\\
     $U_4(2)$ & $2^6 \cdot 3^4 \cdot 5$ & $C_2$ & $(C_2^4 \rtimes C_2) \rtimes C_2$, $C_3^3 \rtimes C_3$, $C_5$ & $A_6$, $A_5$ & \primeline \\
     \hline

\end{tabular}

\caption{$K_3$-groups and some structural information. Relevant subgroups refer to key subgroups which are referenced in proofs in the paper.}
    \label{table:1}
    
\end{table}

The information in Tables \ref{table:2} and \ref{table:3} was computed via GAP \cite{GAP}. 

\begin{table}[h!]
\begin{tabular}{|c||c|}
    \hline
     $G$ & Information about Irr$(G)$ \\ \hline \hline
     $A_6$ & In all representations, elements of order $2, 3,$ or $5$ have a fixed point. \\
     $\PSL(2,17)$ & In all representations, elements of order $2$ or $3$ have a fixed point. \\
     $\SL(2,17)$ & In all representations, elements of order $3$ have a fixed point. \\
     \hline

\end{tabular}

\caption{General information on irreducible complex representations of relevant groups.}
    \label{table:2}
    
\end{table}

\begin{table}[h!]
\begin{tabular}{|c||c|c|c|}
    \hline
     $G$ & $\rho \in \text{Irr}(G)$ & $\deg \rho$ & Information about $\rho$ \\ \hline \hline
     $\PSL(2,7)$ & $\rho_1$ & $3$ & $g \in \PSL(2,7)$ acts fixed point freely if and only if $|g| = 7$ \\
     $\PSL(2,17)$ & $\rho_2$ & $16$ & $g \in \PSL(2,17)$ acts fixed point freely if and only if $|g| = 17$ \\
     \hline

\end{tabular}

\caption{Information on specific irreducible complex representations of relevant groups.}
    \label{table:3}
    
\end{table}

\begin{center}
{\bf Acknowledgements}
\end{center}
\vspace{11pt}

This research was conducted under NSF-REU grant DMS-1757233, DMS-2150205 and NSA grant H98230-21-1-0333, H98230-22-1-0022 during the Summer of 2022 by the first, third, and fourth authors under the supervision of 
the second author. The authors gratefully acknowledge the financial support of NSF and NSA, and also thank Texas State University for providing a great working environment and support. 

\printbibliography

\end{document}